\documentclass[12pt]{amsart}
\usepackage[notref, notcite, final]{showkeys} 
\usepackage[active]{srcltx}
\usepackage{latexsym}
\usepackage{graphicx}
\usepackage{xcolor}
\usepackage{tikz}
\usepackage[normalem]{ulem}
\usepackage{soul}

\usepackage{latexsym}
\usepackage{amsmath}
\usepackage{mathtools}
\usepackage{comment}

\newcommand{\ignore}[1]{}
\renewcommand{\marginpar}[1]{}

\newcommand{\hide}[1]{}

\DeclareMathOperator{\ad}{ad}

\DeclareMathOperator{\Der}{Der}

\newcommand{\F}{F}

\newcommand{\Z}[0]{\mathbb Z}

\newcommand{\ee}{\varepsilon}

\newtheorem{dummy}{Dummy}

\numberwithin{dummy}{section}

\newtheorem{lemma}[dummy]{Lemma}

\newtheorem{theorem}[dummy]{Theorem}

\newtheorem{prop}[dummy]{Proposition}

\newtheorem{notation}[dummy]{Notation}

\theoremstyle{definition}
\newtheorem{definition}[dummy]{Definition}

\newtheorem{convention}[dummy]{Convention}

\theoremstyle{remark}
\newtheorem{rem}[dummy]{Remark}
\newtheorem*{rem*}{Remark to ourselves}

\begin{document}

\bibliographystyle{amsalpha}

\author{M. Avitabile}
\email[M. Avitabile]{marina.avitabile@unimib.it}
\address{Dipartimento di Matematica e Applicazioni\\
  Universit\`a degli Studi di Milano - Bicocca\\
 via Cozzi 55\\
  I-20125 Milano\\
  Italy}
\author{A. Caranti}
\email[A. Caranti]{andrea.caranti@unitn.it}
\address{Dipartimento di Matematica \\
Universit\`a degli Studi di Trento\\
via Sommarive 14\\
I-38123 Trento\\
Italy}

\author{S. Mattarei}
\email[S. Mattarei]{sandro.mattarei@unimib.it}
\address{Dipartimento di Matematica e Applicazioni\\
  Universit\`a degli Studi di Milano - Bicocca\\
 via Cozzi 55\\
I-20125 Milano\\
Italy}

\begin{abstract}
The graded Lie algebra associated with the Nottingham group over a field of prime characteristic serves as  a fundamental example of Nottingham algebras, a class of infinite-dimensional, positively graded thin algebras.

This paper completes the classification of Nottingham algebras initiated in earlier papers, proving both existence and uniqueness results that determine all such algebras up to isomorphism.

\end{abstract}

\title[Classification of {N}ottingham algebras]{Classification of {N}ottingham algebras}

\thanks{The authors are members of INdAM-GNSAGA. Funded by the European Union - Next Generation EU, Missione 4 Componente 1 CUP B53D23009410006, PRIN 2022 - 2022PSTWLB - Group Theory and Applications}

\maketitle

\thispagestyle{empty}

\section{Introduction}

The main result of this paper is a complete classification of {\em Nottingham algebras}.
These are infinite-dimensional, positively graded Lie algebras,
whose most prominent example is the graded Lie algebra associated
with the lower central series
of the {\em Nottingham group}~\cite{Johnson, DdSMS, Camina} over the prime field
$\mathbb{F}_p$, where $p$ is an odd prime.

The structure of a group is reflected in that of a Lie algebra associated with it.
The Nottingham group is a {\em thin group}, that is, a group in which every
nontrivial normal subgroup lies between two suitable terms of its lower central
series and the lower central factors are elementary abelian of order $p$ or $p^2$.
Consequently, its associated Lie algebra is a {\em thin algebra}, that is, a positively graded Lie algebra  where each homogeneous component has dimension one or two, and which satisfies the following {\em covering property:}
every nonzero graded ideal lies between two suitable consecutive Lie powers of the algebra.
Homogeneous components of dimension two in a thin algebra are called {\em diamonds}, a term
originating from  their lattice of ideals.
The first homogeneous component of a thin algebra is always a diamond,
the {\em first diamond} of the algebra.

The second  diamond of the Lie algebra associated with the Nottingham group has degree $p$, reflecting the fact that the non-cyclic lower central factors
of the group occur in each degree congruent to $1$ modulo $p-1$.
The Lie algebra associated with the Nottingham group has been extensively studied in~\cite{Car:Nottingham}, where a natural generalization is presented in which the second diamond has degree a power of the characteristic of the underlying field. A thin algebra $L =\bigoplus_{i=1}^{\infty}L_i$ over a field of characteristic  an odd prime $p$,  is a Nottingham algebra if its second diamond occurs in degree $q$, where $q$ is a power of $p$.

In every Nottingham algebra, each diamond after the first can be assigned a {\em type},
which is an element of the underlying field or, possibly, $\infty$.
The second diamond $L_q$ has invariably type $-1$, and we assign no type to the
first diamond $L_1$. For a couple of values of the type, the corresponding
diamonds are degenerate and  we refer to them as {\em fake diamonds}.
The type of a diamond $L_m$ describes the adjoint action of $L_1$ on $L_m$, in such a way that knowledge of
all degrees in which diamonds occur in $L$, and their types, determines $L$ up to isomorphism. We recall the definition of type of a diamond in Section~\ref{sec:types}.
The difference in degrees of any two consecutive diamonds in a Nottingham algebra
has been investigated in~\cite{AviMat:diamond_distances}, where it is proved
that it equals $q-1$, up to an appropriate interpretation of fake diamonds.
For the reader's convenience, in Theorem~\ref{thm:distance} we quote the
results of~\cite{AviMat:diamond_distances}, formulated in a way that suits
the purposes of this paper.

A wide variety of Nottingham algebras are known. In several cases, the types
of the diamonds follow a periodic pattern. We refer to
those algebras as {\em regular} Nottingham algebras.
They arise from certain cyclic gradings of simple Lie algebras of Cartan type.
In particular, the thin algebra associated with the Nottingham group is a regular
one and arises from a cyclic grading of the {\em Witt algebra}.
We give a more detailed description in Theorem~\ref{thm:regular}, which is an 
existence result for Nottingham algebras with periodic diamond-type patterns and
summarizes the conclusions of various papers (\cite{Car:Nottingham}, \cite{Avi}, \cite{AviMat:A-Z}, \cite{AviMat:mixed_types}).
Uniqueness results are also obtained for the algebras in Theorem~\ref{thm:regular},
in the sense that each of them is uniquely determined by an appropriate
finite-dimensional quotient.
We refer the reader to the commentaries following Theorem~\ref{thm:regular} for details.

Among Nottingham algebras, regular algebras form a small minority.
Further Nottingham algebras, and in fact uncountably many,
are closely related to (graded Lie) algebras of {\em maximal class}.
They have been described by David Young  in his  PhD thesis~\cite{Young:thesis}.
The original aim of his work was the classification of Nottingham algebras with
a fake third diamond in degree $2q-1$.
Algebras of {\em coclass} $2$ form a special case of these.
Another special class is formed  by the algebras
$\mathcal{T}_{q,1}(M)$, which are in one-to-one correspondence with the
algebras of maximal class having at most two distinct {\em two-step centralizers}.
His goal was fully achieved, and we quote his classification result as 
Theorem~\ref{thm:long_second_chain}.

Although it was beyond the scope of his work, Young also described
a second procedure that produces a Nottingham algebra from an algebra
of maximal class, which are denoted by $\mathcal{T}_{q,2}(M)$.
In this paper, we address the existence and uniqueness of these algebras,
thereby completing the final step toward the full classification of Nottingham algebras.
In Section~\ref{sec:construction} we show that the algebras  $\mathcal{T}_{q,2}(M)$ 
are in one-to-one correspondence with the algebras of maximal class having
two distinct two-step centralizers.
We prove in Theorem~\ref{thm:uniqueness} that these algebras are uniquely determined by a suitable finite-dimensional quotient.
A crucial role in our argument is played by a natural $(\mathbb{Z}\times\mathbb{Z})$-grading of a Nottingham algebra, which we introduce in Section~\ref{sec:grading}.
To streamline the exposition, the most technical lemmas are deferred
to the concluding Section~\ref{sec:gc}, with forward references provided as needed throughout the paper.

Together, Theorems~\ref{thm:regular} and~\ref{thm:irregular}
provide a complete classification of Nottingham algebras,
thereby settling the classification problem.

\section{Nottingham algebras and their diamonds}\label{sec:types}

This section, which introduces basic definitions and results on Nottingham algebras, is essentially quoted from~\cite{AviMat:diamond_distances} and~\cite{AviMat:earliest}.
(Note that the logical order of these two articles differs from their chronological order.)

A {\em thin} (Lie) algebra is a graded Lie algebra $L=\bigoplus_{i=1}^{\infty}L_i$,
with $\dim L_1=2$ and satisfying the {\em covering property:}
for each $i$, each nonzero $z\in L_i$ satisfies $[zL_1]=L_{i+1}$.
However, in this paper we assume $L$ to have infinite dimension,
and it follows that $L$ has trivial centre.
As we mentioned in the Introduction, our definition of a Nottingham algebra includes restrictions on $p$ and $q$.

\begin{definition}\label{def:Nottingham}
A {\em Nottingham algebra} is a thin algebra, over a field of characteristic $p>3$,
with second diamond $L_q$, where $q>5$ is a power of $p$.
\end{definition}

The restrictions on $p$ and $q$ in Definition~\ref{def:Nottingham} serve to avoid
various instances of exceptional behaviour in small characteristics,
which start where certain Nottingham algebras
fit the recipe of other families of thin  algebras.

In this paper we use the left-normed convention
for iterated Lie products, hence $[abc]$ stands for $[[ab]c]$.
We also use the shorthand
$[ab^i]=[ab\cdots b]$,
where $b$ occurs $i$ times.

Let us start a bit more generally and consider first a thin  algebra $L$ with $\dim(L_3)=1$.
Then there is a nonzero element $y$ of $L_1$, unique up to a scalar multiple, such that $[L_2y]=0$.
Extending to a basis $x,y$ of $L_1$, this means $[yxy]=0$.
It is not hard to deduce from this relation that no two consecutive components in such a thin  algebra $L$
can both be diamonds, see~\cite{Mat:sandwich} for a proof.
Thus, any two diamonds are separated by one or more one-dimensional homogeneous components.

More delicate arguments in~\cite{Mat:sandwich}
show that any  Lie algebra $L$
with second diamond occurring past $L_5$
satisfies $[Lyy]=0$, which means $(\ad y)^2=0$.
According to a well-known definition of Kostrikin's,
in odd characteristic that means $y$
is a {\em sandwich element} of $L$.
The significance of this fact is discussed in~\cite{Mat:sandwich}.
In particular, Nottingham algebras as defined in this paper satisfy $(\ad y)^2=0$.
One of our reasons for excluding $q=5$ from our definition of a Nottingham algebra
is the existence of the thin  algebra of classical type $A_2$ of~\cite{CMNS} (and~\cite{Mat:thin-groups})
which has a diamond in each degree congruent to $\pm 1$ modulo $6$ (in any characteristic except $2$).
In particular, this thin algebra has $L_5$ and $L_7$ as second and third diamond, which easily implies $[L_5yy]\neq 0$.

From now on let $L$ be a Nottingham algebra with second diamond $L_q$.
Then the element $y$ centralizes each homogeneous component from $L_{2}$ up to $L_{q-2}$.
That is a nontrivial assertion proved in~\cite{CaJu:quotients}, and relies on the theory of algebras of maximal class established in~\cite{CMN,CN}.
Consequently, $L_i$ is spanned by $[yx^{i-1}]$ for $2\le i<q$.
In particular, $v_1=[yx^{q-2}]$ spans the component $L_{q-1}$ and, in turn, $[v_1x]$ and $[v_1y]$ span the second diamond $L_q$.
The subscript in $v_1$ is motivated by a convenience
of numbering the diamonds in a certain way,
and this notation turns out to be more natural than denoting it $v_2$ as in some earlier papers.

It is now easy to see that one may redefine $x$ in such a way that
\begin{equation}\label{eq:second_diamond_type}
[v_1xx]=0=[v_1yy] \quad\textrm{and} \quad [v_1yx]=-2[v_1xy],
\end{equation}
see~\cite[Section~3]{AviMat:A-Z} for a cleaner excerpt of the original argument in~\cite{Car:Nottingham}.
In the rest of this paper we refer to such $x$ and $y$ as {\em standard generators} of $L$.

It follows immediately from Equation~\eqref{eq:second_diamond_type}
that the subspaces $\F x$ and $\F y$ are uniquely determined. The type of a diamond, as introduced in Definitions~\ref{def:type} and~\ref{def:type-fake} below, will not be affected by the choice of the generators of these two subspaces.

Because $[yx^q]=[v_1xx]=0$, we have $(\ad x)^q=0$.
Indeed, since $(\ad x)^q$ is a derivation of $L$, its kernel is a subalgebra, but then that must equal $L$ as both generators $x$ and $y$ belong to it.

We recall the definition of {\em type} of a diamond as introduced in ~\cite{CaMa:Nottingham}.
(Note that diamond types are defined differently for thin  algebras with second diamond $L_{2q-1}$, see~\cite{CaMa:thin}.)
We do not assign a type to the first diamond $L_1$.
Now let $L_m$ be a diamond past $L_1$, that is, a two-dimensional
homogeneous component of $L$ with $m>1$.
Because no two consecutive homogeneous components can be diamonds, $L_{m-1}$
is one-dimensional, and so is $L_{m+1}$.
If $w$ spans $L_{m-1}$, then $L_m$ is spanned by $[wx]$ and $[wy]$,
and $L_{m+1}$ is spanned by $[wxx]$, $[wxy]$, $[wyx]$ and $[wyy]$.
The following definition encodes particular relations
among the latter four elements.

\begin{definition}\label{def:type}
Let $L$ be a Nottingham algebra,
with second diamond $L_q$ and standard generators $x$ and $y$.
Let $L_m$ be a diamond of $L$, with $m>1$,
and let $w$ be a nonzero element in $L_{m-1}$.
\begin{itemize}
\item[(a)]
We say $L_m$ is a diamond of {\em finite type} $\mu$, where $\mu\in\F$, if
\begin{equation*}
[wxx]=0=[wyy] \quad\textrm{and} \quad \mu[wyx]=(1-\mu)[wxy].
\end{equation*}
\item[(b)]
We say $L_m$ is a diamond of {\em infinite type} if
\begin{equation*}
[wxx]=0=[wyy] \quad\textrm{and} \quad [wyx]=-[wxy].
\end{equation*}
\end{itemize}
\end{definition}

In particular, this definition applies to the second diamond $L_{q}$, which therefore has invariably type $\mu=-1$.

Note that the relations of Case~(a) read
\begin{equation*}
[wyy]=0,\quad [wzz]=[wyz]+[wzy], \quad [wzy]=\mu[wzz]
\end{equation*}
when stated in terms of the alternate generators $x$ and $z=x+y$ for $L$.
This was the choice of generators for $L$ made for calculations in~\cite{CaMa:Nottingham},
and it had the advantage that $[yz^{k-1}]$ is a nonzero element of $L_k$ for all $k>0$, and thus
$[yz^{k}]$ and $[yz^{k-1}y]$ span $L_{k+1}$.
The standard generators $x$ and $y$ employed here became preferable in later developments,
in great part because the diamond relations are homogeneous in the pair $x,y$,
and thus Nottingham algebras acquire a double grading that will be introduced in Section~\ref{sec:grading}.

The values $\mu=0$ and $\mu=1$ cannot actually occur in
Definition~\ref{def:type}.
If $\mu=0$ then the relations $[wxx]=0=[wxy]$
would imply that the element $[wx]$ is central,
and hence vanish because of our blanket assumption $\dim(L)=\infty$.
Similarly, if $\mu=1$ then the element $[wy]$ would be central,
and hence vanish.
Thus, strictly speaking, diamonds of type $0$ or $1$ cannot occur,
at least if we insist that a diamond should have dimension two,
as in Definition~\ref{def:type}.
Nevertheless, it is convenient for a uniform description
of the diamonds patterns
in Nottingham algebras to allow ourselves to
call {\em diamonds of type $0$ or $1$}
those one-dimensional homogeneous components $L_m$ that satisfy the relations of
Definition~\ref{def:type}
with $\mu=0$ or $1$.
This leads us to the following definition.

\begin{definition}\label{def:type-fake}
Let $L$ be a Nottingham algebra,
with second diamond $L_q$ and standard generators $x$ and $y$.
Let $L_{m-1}$ be a one-dimensional component, spanned by $w$, with $m>1$.
\begin{itemize}
\item[(a)]
We say $L_m$ is a diamond of {\em of type $1$} if
\begin{equation*}
[wxx]=0 \quad\textrm{and} \quad [wy]=0.
\end{equation*}
\item[(b)]
We say $L_m$ is a diamond {\em of type $0$} if
\begin{equation*}
[wx]=0 \quad\textrm{and} \quad [wyy]=0.
\end{equation*}
\end{itemize}
\end{definition}
We refer to diamonds of type $0$ or $1$ as {\em fake diamonds}
to distinguish them from the {\em genuine diamonds} of dimension two.
The necessity of including fake diamonds in a treatment
of Nottingham algebras arose, early in the development,
from the fact that in various notable instances diamonds occur at regular intervals, with types following an arithmetic progression
(see Section~\ref{sec:regular}).
When such arithmetic progression of types passes through $0$ or $1$, fake diamonds occur.

However, this carries an inherent ambiguity:
whenever $L_m$ satisfies the definition of a diamond of type $1$
(which amounts to $[L_{m-1}y]=0$ and $[L_mx]=0$), the next homogeneous component $L_{m+1}$ satisfies the definition of a diamond of type $0$
(because $[L_my]=L_{m+1}$ due to the covering property, and then $[L_{m+1}y]=0$ due to $[Lyy]=0$).
Thus, if $w$ spans $L_{m-1}$ then $[wx]$ spans $L_m$
and $[wxy]$ spans $L_{m+1}$, and we have the relations
\begin{equation}\label{eq:fake_diamonds}
[wy]=0, \qquad
[wxx]=0, \qquad
[wxyy]=0.
\end{equation}
The first and second are those in part~(a) of
Definition~\ref{def:type-fake}, and the second and third
are those in part $(b)$ if we use $w'=[wx]$ instead of $w$ in it.
For various reasons it is inconvenient to simultaneously regard
two consecutive components as fake diamonds, and so we adopt the following convention.
A motivation for this convention can be seen in Theorem~\ref{thm:distance},
and a concrete example appears in Subsection~\ref{subsec:L(q)}.

\begin{convention}\label{convention:fake}
Whenever we have a diamond $L_m$ of type $1$, necessarily followed by a diamond $L_{m+1}$ of type $0$,
we allow ourselves to call (fake) diamond precisely one of $L_m$
and $L_{m+1}$, of the appropriate type, and not the other.
\end{convention}

It is not at all obvious that a two-dimensional component
$L_m$ of an arbitrary Nottingham algebra, with $m>q$,
should satisfy the relations of Definition~\ref{def:type} for some $\mu$, thus allowing
type $\mu$ to be assigned to it.
In fact, this is one of the main conclusions of~\cite{AviMat:diamond_distances}.
More generally, it is shown there that whenever
$[L_{m-2}y]=0$ and $[L_{m-1}y]\neq 0$ for some $m>q$,
either $L_m$ is two-dimensional and can be assigned a
type $\mu$ according to Definition~\ref{def:type},
or $L_m$ is a (fake) diamond of type $0$
according to Definition~\ref{def:type-fake}.
As we have observed right after that definition, the latter situation
admits the alternate interpretation that $L_{m-1}$ is a diamond of type $1$.

Another consequence of the results of~\cite{AviMat:diamond_distances}
is that any two consecutive diamonds in an arbitrary Nottingham algebra can always be assumed to have a difference of $q-1$
in degrees, but only allowing an appropriate interpretation in case fake diamonds are involved, according to Convention~\ref{convention:fake}.
We quote what we need in a formulation given in~\cite{AviMat:earliest}.

\begin{theorem}[Theorem~3.1 in~\cite{AviMat:earliest}]\label{thm:distance}
Let $L$ be a Nottingham algebra with second diamond $L_q$, and standard generators $x$ and $y$.
Let $L_m$ be a (possibly fake) diamond of $L$, with $m\ge q$.
\begin{itemize}
\item[(a)]
If $L_m$ is a genuine diamond then $L_{m+q-1}$ is a diamond.
\item[(b)]
If $L_m$ is a diamond of type $1$, then either $L_{m+q-1}$
or $L_{m+q}$ is a diamond.
\item[(c)]
If $L_m$ is a diamond of type $0$ and, in addition,
$L_{m-q+1}$ is a diamond of nonzero type,
then $L_{m+q-1}$ is a diamond.
\item[(d)]
In each of the previous cases $y$ centralizes $L_{m+1},\ldots,L_{m+q-3}$,
and also $L_{m+q-2}$ if $L_{m+q-1}$ is not a diamond in assertion~(b).
\end{itemize}
\end{theorem}
Various clarifications are in order.
Assertion~(d) of
Theorem~\ref{thm:distance},
together with~(a) and~(b),
show that in a Nottingham algebra
$y$ centralizes each homogeneous
component which is not
a diamond or immediately
precedes a diamond.
Thus, the degrees and types of the diamonds
(still making use of Convention~\ref{convention:fake} on fake diamonds)
suffice to completely describe an
arbitrary Nottingham algebra.

Next, the two conclusions of assertion~(b)
are not disjoint, the common case being
when $L_{m+q-1}$ is a diamond of type $1$,
which means the same as $L_{m+q}$
being a diamond of type $0$.
Also, the hypothesis of assertion~(b)
can be alternately read as $L_{m+1}$
being a diamond of type $0$.
Altogether, we see that the difference
in degree between consecutive diamonds
can always be read as $q-1$,
as long as we suitably interpret
any fake diamonds involved.

Finally, note that if we read the hypothesis of assertion~(c)
as $L_{m-1}$ being a diamond of type $1$, then assertion~(b)
would only tell us that either $L_{m+q-2}$ or $L_{m+q-1}$
is a diamond.
In fact, inferring the stronger
conclusion of assertion~(c)
requires information on the diamond
which precedes $L_m$.

All assertions of Theorem~\ref{thm:distance} were stated
in~\cite{AviMat:diamond_distances} under a blanket hypothesis
that Nottingham algebras have infinite dimension.
However, as pointed out there, those assertions remain true
for a finite-dimensional Nottingham algebra $L$
as long as none of the homogeneous components
they mention is the last or next-to-last nonzero
homogeneous component of $L$.

\section{The double grading of a Nottingham algebra}\label{sec:grading}

The fact, established in~\cite{AviMat:diamond_distances},
that every diamond of a Nottingham algebra can be assigned a type
(possibly including fake diamonds of types $1$ or $0$ as described above)
has the remarkable consequence that Nottingham algebras are {\em bigraded,} as follows.
If $L$ is a Nottingham algebra with standard generators $x$ and $y$, then
a natural $(\Z\times\Z)$-grading
$\bigoplus_{r,s\in\Z}L_{(r,s)}$
of $L$ is assigned by declaring $x$ and $y$
to have degrees $(1,0)$ and $(0,1)$, respectively.
We have
$[L_{(r,s)},x]\subseteq L_{(r+1,s)}$
and
$[L_{(r,s)},y]\subseteq L_{(r,s+1)}$,
and this grading is well defined because of the relations that hold in a Nottingham algebra,
notably the relations that define diamond types.

It might aid intuition to draw (partial) diagrams of $L$ according to this grading,
where, say, the vector $(1,0)$ points South-West and the vector $(0,1)$ points South-East.
Thus, the structure of $L$ up (or rather down) to its second diamond $L_q$
looks like in the following diagram.

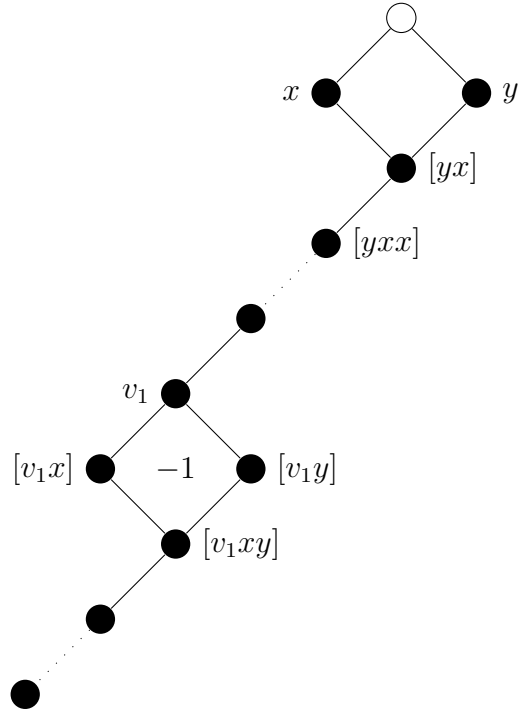
\begin{center}
\begin{tikzpicture}
\node[draw, circle, minimum size=0.2cm] (CircleA) at (0,0) {};
    \node[fill, circle, minimum size=0.2cm] (CircleB) at (-1,-1) [label=left:$x$] {};
        \node[fill, circle, minimum size=0.2cm] (CircleC) at (1,-1)
        [label=right:$y$] {};
                \node[fill, circle, minimum size=0.2cm] (CircleD) at
                (0,-2)
                     [label=right:\protect{$[yx]$}] {};
                \node[fill, circle, minimum size=0.2cm] (CircleE) at
                (-1,-3)
                     [label=right:\protect{$[y x x]$}] {};
                \node[fill, circle, minimum size=0.2cm] (CircleF) at (-2,-4) {};

    \node[fill, circle, minimum size=0.2cm] (CircleA1) at (-3,-5)
    [label=left:$v_{1}$] {};
   \node[fill, circle, minimum size=0.2cm] (CircleB1) at (-4,-6)
   [label=left:\protect{$[v_{1} x]$}] {};
        \node[fill, circle, minimum size=0.2cm] (CircleC1) at (-2,-6)
           [label=right:\protect{$[v_{1} y]$}] {};
                \node[fill, circle, minimum size=0.2cm] (CircleD1) at (-3,-7)
                     [label=right:\protect{$[v_{1} x y]$}] {};
                     \node (text-1) at (-3, -6.45) [label=$-1$] {};
              \node[fill, circle, minimum size=0.2cm] (CircleE1) at (-4,-8) {};
                \node[fill, circle, minimum size=0.2cm] (CircleF1) at (-5, -9) {};


    \draw (CircleD) -- (CircleB) -- (CircleA) -- (CircleC) -- (CircleD)
    -- (CircleE);
    \draw[loosely dotted] (CircleE) -- (CircleF) ;
    \draw (CircleF) -- (CircleA1);
    \draw (CircleD1) -- (CircleB1) -- (CircleA1) -- (CircleC1) -- (CircleD1)
    -- (CircleE1);
        \draw[loosely dotted] (CircleE1) -- (CircleF1) ;
\end{tikzpicture}
\end{center}
Each node in the diagram stands for a nonzero homogeneous component of $L$
in the double grading, and each of those is one-dimensional,
except for the top node, which is drawn just for reference as it represents the component
$L_{(0,0)}=\{0\}$.
Although the term {\em diamond}, assigned to a two-dimensional component
of a thin  algebra, originated from a description of the lattice of ideals,
it somehow fits this graphical depiction as well.
The number $-1$ placed inside the second diamond indicates its type.

As discussed after Definition~\ref{def:type-fake}, leading
to Convention~\ref{convention:fake}, in principle the homogeneous component
of $L$ that immediately follows a fake diamond of type $1$ satisfies the definition
of a diamond of type $0$.
This is illustrated by the following local picture near a fake diamond.

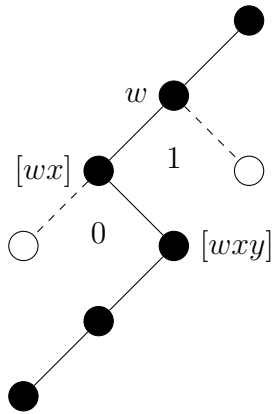
\begin{center}
        \begin{tikzpicture}[scale=1.0]
        \node[fill,draw, circle, minimum size=0.2cm] (DotA) at (0,0) {};
        \node[fill,draw, circle, minimum size=0.2cm] (DotB) at (-1,-1)
             [label=left:$w$] {};
             \node[fill,draw, circle, minimum size=0.2cm] (DotC) at (-2,-2)
                  [label=left:\protect{$[w x]$}] {};
        \node[draw, circle, minimum size=0.2cm] (DotD) at (-3,-3) {};

             \node[draw, circle, minimum size=0.2cm] (DotA1) at (0,-2) {};
        \node[fill,draw, circle, minimum size=0.2cm] (DotB1) at
        (-1,-3)
             [label=right:\protect{$[w x y]$}] {};
        \node[fill,draw, circle, minimum size=0.2cm] (DotC1) at (-2,-4) {};
        \node[fill,draw, circle, minimum size=0.2cm] (DotD1) at
        (-3,-5) {};

        \node (text1) at (-1,-2.25) [label=$1$] {};
       \node (text0) at (-2,-3.25) [label=$0$] {};

 \draw (DotA) -- (DotB) -- (DotC) -- (DotB1) ;
     \draw[dashed] (DotC) -- (DotD) ;

     \draw  (DotB1) -- (DotC1) -- (DotD1) ;
     \draw[dashed] (DotB) -- (DotA1) ;

        \end{tikzpicture}\end{center}

The dashed lines in the picture lead to elements $[wy]$ and $[wxy]$
that are depicted here to reinforce the notion of type of a fake diamond, but are actually zero in $L$.
However, those elements would be nonzero in certain natural central extensions that occur
in the construction of certain Nottingham algebras from finite dimensional algebras
in~\cite{Avi,AviMat:A-Z,AviMat:mixed_types}, occasionally dubbed {\em quasithin}.

We conveniently introduce the {\em support} of a Nottingham algebra $L$ as the set of bidegrees $(r,s)$
such that $L_{(r,s)}$ is nonzero, and thus one-dimensional.
This allows for a simplification of certain calculations in $L$,
which we call the {\em support argument}:
if the result of any Lie bracket calculation in $L$ is known to be a homogeneous element in the double grading
(usually because it is the result of successive Lie brackets performed on homogeneous elements),
and its bidegree lies outside the support of $L$, then we conclude that the result must be zero
without any need to go through the details of the calculation.
The same conclusion can often be drawn even if the expected bidegree of the final result lies in the support of $L$,
as long as one is forced to leave the support at some intermediate stage of the calculation.

Before we clarify the support argument with an example, this seems a good place to introduce a general calculation tool,
the {\em generalized Jacobi identity}
\[
[a[bc^n]]=\sum_{i=0}^{n} (-1)^i \binom{n}{i} [ac^{i}bc^{n-i}].
\]
Two special instances which often occur are $[a[bc^q]]=[abc^q]-[ac^qb]$ (which amounts to $(\ad c)^q$ being a derivation), and
$
[a[bc^{q-1}]]=\sum_{i=0}^{q-1}\;[ac^{i}bc^{q-1-i}],
$
due to $\binom{q-1}{i}\equiv (-1)^i\pmod{p}$.
More generally, the binomial coefficients involved in the generalized Jacobi identity can be efficiently evaluated modulo $p$ by means of Lucas' theorem:
if $q$ is a power of $p$ and  $a,b,c,d$ are non-negative
integers with $b,d<q$, then
$
\displaystyle \binom{aq+b}{cq+d}\equiv \binom{a}{c}\binom{b}{d} \pmod p.
$

We illustrate the support argument by reviewing the proof of a simple but basic fact
on Nottingham algebras.
Suppose that we know the structure of a Nottingham algebra up to degree $q$,
where we expect the second diamond to lie.
Hence the support of $L$ up to degree $q$ consists of the bidegree $(1,0)$, the bidegrees $(i,1)$ with $0\le i<q$, and the bidegree $(q-2,2)$.
Then the following calculation establishes that the second diamond must have type $-1$:
\begin{align*}
0&=[yx^{(q-1)/2}[yx^{(q-1)/2}]]
=\sum_{i=0}^{(q-1)/2} (-1)^i \binom{(q-1)/2}{i} [yx^{(q-1)/2+i}yx^{(q-1)/2-i}]
\\&=
\pm\left(\binom{(q-1)/2}{(q-3)/2} [v_1yx]
-\binom{(q-1)/2}{(q-1)/2} [v_1xy]\right)
\\&=
\pm\left(\frac{q-1}{2} [v_1yx]
-[v_1xy]\right),
\end{align*}
whence
$[v_1yx]=-2[v_1xy]$.
Here each term of the summation except for the last two vanishes because some initial portion of the iterated Lie bracket falls outside
the support.

The support of several families of Nottingham algebras $L$ (namely, those described in Section~\ref{sec:regular})
has a simple description in terms of a
pair of inequalities that depend on the parameter $q$ of $L$.
We set
\[
S_\le(q)=\{(r,s)\in\Z^2:r,s>0,\ -1\le(q-2)s-r\le q-2\},
\]
and $S_<(q)$ defined similarly but with strict inequalities.
We say that a Nottingham algebra $L$ is {\em regular}
if its support is contained in $S_\le(q)$, and {\em irregular} otherwise.

As it turns out, the irregular Nottingham algebras are
the algebras $\mathcal{N}(q,r)$, $\mathcal{T}_{q,1}(M)$ and $\mathcal{T}_{q,2}(M)$, constructed in David Young's thesis~\cite{Young:thesis}
from certain algebras of maximal class $M$.
We will describe Young's constructions in Section~\ref{sec:further}.
Any other Nottingham  algebra $L$ with parameter $q$ has support containing
$S_{<}(q)$ and contained in $S_{\le}(q)$.
If, in addition, $L$ has no fake diamonds, then its support equals $S_{\le}(q)$.

We conclude this section with noting that, because a Nottingham algebra $L$ is finitely generated,
the algebra of derivations $\Der(L)$ inherits both a $\Z$-grading and
a $(\Z\times\Z)$-grading from the natural grading and the double grading of $L$.
A bigraded derivation of $L$ that is not inner will play a role in
Proposition~\ref{prop:D}, where $L$ belongs to a special class of Nottingham algebras.

\section{Regular Nottingham algebras}\label{sec:regular}

In this section we present a description of all the regular Nottingham algebras that has been achieved over
a number of papers, and discuss a portion of their classification.
The material of this section is lifted from~\cite{AviMat:diamond_distances}.

In various cases where fake diamonds occur there is a natural choice, between
calling $L_m$ a diamond of type $1$, or $L_{m+1}$ a diamond of type $0$,
that makes all diamonds (including the fake ones) occur at regular distances
with a difference of $q-1$ in degrees;
this was formalized in Section~\ref{sec:grading} as they being regular.
We illustrate that through the following existence result, which will be clarified and
supplemented with uniqueness statements in commentaries to follow.

\begin{theorem}[Theorem~7 in ~\cite{AviMat:diamond_distances}]\label{thm:regular}
There exist Nottingham algebras $L$ with second diamond $L_q$,
where (possibly fake) diamonds occur in  each degree congruent to $1$ modulo $q-1$, and have types described by any of the following patterns:
\begin{itemize}
\item[(a)]
all diamonds of type $-1$~\cite{Car:Nottingham};
\item[(b)]
all diamonds of finite types following a non-constant arithmetic progression~\cite{Avi,AviMat:A-Z};
\item[(c)]
all diamonds of infinite type except for those in degrees $\equiv q\pmod{p^s(q-1)}$ for some $s>0$, which have type $-1$~\cite{AviMat:A-Z};
\item[(d)]
all diamonds of infinite type except for those in degrees $\equiv q\pmod{p^s(q-1)}$ for some $s>0$, which have finite types following a non-constant arithmetic progression~\cite{AviMat:mixed_types};
\item[(e)]
all diamonds of infinite type except for $L_q$.
\end{itemize}
\end{theorem}

Note that all algebras of Theorem~\ref{thm:regular} exhibit a periodic pattern, not only in terms of diamond distances, but also of their types;
this periodicity is stronger than those algebras just being regular.

Nottingham algebras as in case~(a) of Theorem~\ref{thm:regular},
thus with all diamonds having type $-1$,
were explicitly constructed in~\cite{Car:Nottingham},
using a certain cyclic grading of Zassenhaus algebras.
The special case where $q=p$ is the graded Lie algebra associated with the lower central series of the Nottingham group, thus justifying their name.
They were also shown in~\cite{Car:Nottingham} to be uniquely determined by some finite-dimensional quotient.
Here and in certain other cases such `uniqueness' was proved by exhibiting a finite presentation for some central extension of $L$.
(In most cases $L$ is not itself finitely presented.)

Regarding case~(b) of Theorem~\ref{thm:regular},
Nottingham algebras including fake diamonds were first observed in~\cite{CaMa:Nottingham}).
More precisely, finite presentations for certain central extensions (and second-central in one case) of Nottingham algebras were given
(with a couple of exceptions due to the existence of algebras
$\mathcal{L}_{1,q}$, $\mathcal{N}(q,r)$, $\mathcal{T}_{q,1}(M)$,
which we discuss in Section~\ref{sec:further}),
where the diamonds occur in all degrees congruent to $1$ modulo $q-1$,
and their types follow a non-constant arithmetic progression.
If that arithmetic progression passes through $0$, that is, if it runs through the prime field,
then those diamonds include fake diamonds, of both types $0$ and $1$.
Such Nottingham algebras were explicitly constructed, thus proving their existence,
in~\cite{Avi} in case all types belong to the prime field, and in~\cite{AviMat:A-Z} otherwise.
Those constructions used certain finite-dimensional simple modular Lie algebras of Cartan type,
and certain gradings of them over a finite cyclic group.

Nottingham algebras where the third diamond has infinite type include those of cases~(c) and~(d).
Again, their constructions in~\cite{AviMat:A-Z} and~\cite{AviMat:mixed_types} used
certain finite-dimensional simple modular Lie algebras of Cartan type, but special tools involving generalized exponentials of derivations
had to be developed for producing the required gradings,
in~\cite{Mat:Artin-Hasse,AviMat:Laguerre,AviMat:gradings}.
Uniqueness of those Nottingham algebras was established in~\cite{AviMat:earliest}.

The algebras of Case~(e) can be obtained from those of Case~(c) or~(d)
by letting the parameter $s$ go to infinity;
this can be formalized by a standard inverse limit construction.

In all cases of Theorem~\ref{thm:regular}, each homogeneous component which is not a diamond or immediately precedes a diamond is centralized by $y$.
Together with this information,
the locations and types of all diamonds,
as specified in each case of Theorem~\ref{thm:regular},
give a complete description of those Nottingham algebras.
Note that each of the Nottingham algebras of Theorem~\ref{thm:regular} has diamonds in each degree congruent to $1$ modulo $q-1$.
In particular, the distance between consecutive diamonds in those particular Nottingham algebras
is invariably $q-1$, provided that we assign
an appropriate type $0$ or $1$ to each fake diamond,
making use of Convention~\ref{convention:fake}.

\section{The other families of Nottingham algebras}\label{sec:further}

We mentioned in Section~\ref{sec:regular} that uniqueness of all Nottingham algebras $L$
of Theorem~\ref{thm:regular} has been established in all cases except for two instances,
belonging to Case~(b) of Theorem~\ref{thm:regular}.
Recall that in that case of Theorem~\ref{thm:regular} the diamonds of $L$ occur in all degrees congruent to $1$ modulo $q-1$, and their types follow an arithmetic progression.
Such arithmetic progression is uniquely determined by the type of the third diamond $L_{2q-1}$
(which must be finite, otherwise one ends up in Cases (c) or (d) of Theorem~\ref{thm:regular}).
However, when that type is $1$ or $0$, and hence the diamond $L_{2q-1}$ is fake,
$L$ is not uniquely determined as a Nottingham algebra by those prescriptions.

In fact, there are plenty of Nottingham algebras that have no genuine diamond in degree $2q-1$;
in a sense those are actually the majority of Nottingham algebras.
In his PhD thesis~\cite{Young:thesis} David Young described all of them, and managed to prove a classification result
that we quote as Theorem~\ref{thm:long_second_chain}.
In the rest of this section we present concise descriptions of Young's algebras.

\subsection{The algebras $\mathcal{L}_{1,q}$ and $\mathcal{L}_{0,q}$}\label{subsec:L(q)}\label{subsec:Ell}

For each power $q$ of the characteristic, two Nottingham algebras
$\mathcal{L}_{1,q}$ and $\mathcal{L}_{0,q}$ were described in~\cite{Young:thesis}
that have precisely two genuine diamonds, namely $L_1$ and $L_q$, the latter necessarily of type $-1$.
Past their second diamond $L_q$, each of them has a fake diamond in each degree $m>q$ with $m\equiv -1\pmod{q}$,
but those fake diamonds are all of type $1$ in case of $\mathcal{L}_{1,q}$,
and all of type $0$ in case of $\mathcal{L}_{0,q}$.
Thus, in both algebras we observe distances of $q$ rather than $q-1$ between certain fake consecutive diamonds of the same type (either both $0$ or both $1$).
However, we may use the ambiguity of fake diamonds
expressed in Convention~\ref{convention:fake} to our advantage.
For example, in the case of $\mathcal{L}_{1,q}$
(the other case being analogous),
suppose $L_{m-q}$, $L_{m}$ and $L_{m+q}$ are consecutive fake diamonds of type $1$.
We may reinterpret the first of these three fake diamonds
as $L_{m-q+1}$ having type $0$, which has distance $q-1$
from $L_{m}$, fake of type $1$.
Then we may reinterpret the middle diamond as
$L_{m+1}$ having type $0$, which has distance $q-1$
from $L_{m+q}$, fake of type $1$.
Thus, allowing this double interpretation of the middle diamond
avoids contradicting the general statement that any two given consecutive diamonds in a
Nottingham algebra have distance $q-1$ when suitably computed in case of fake diamonds.

We might mention that  $\mathcal{L}_{1,q}$ and $\mathcal{L}_{0,q}$
are algebras of coclass two.
The {\em coclass} of a finitely generated, residually nilpotent Lie algebra is
\[
\sup\{\dim(L/L^{k+1})-k:k\ge 1\}.
\]
With this definition, the algebras of maximal class
(see Section~\ref{sec:construction} for details)
are precisely the algebras of coclass one.

\subsection{The Nottingham deflations $\mathcal{N}(q,r)$}\label{subsec:N(q,r)}
For a (positively) graded Lie algebra $L$ one may define its {\em deflation} $L^{\downarrow}$ as the algebra generated by $(\ad L_1)^p$ and $L_p$,
with its degree function rescaled by dividing each degree by $p$.
This tool was introduced in~\cite{CMN} for the study of  algebras of maximal class.

Now suppose $L$ is $\mathcal{N}(q)$, the Nottingham algebra introduced in~\cite{Car:Nottingham}.
In the case $q=p$, where $N(p)$ is simply the graded Lie algebra associated to the lower central series of the Nottingham group over the field $\mathbb{F}_p$,
we have $N(p)^{\downarrow}\cong N(p)$.
However, if $q>p$ then $\mathcal{N}(q)^{\downarrow}$ is a Nottingham algebra with second diamond in degree $q/p$, which Young denoted by $\mathcal{N}(q/p,p)$,
and will only be isomorphic to $\mathcal{N}(q/p)$ when $q=p^2$.
More generally, applying deflation $j$ times to $\mathcal{N}(qr,1)=\mathcal{N}(qr)$, where $r=p^j>1$, gives an algebra $\mathcal{N}(q,r)$, which has second diamond in degree $q$.
This repeated deflation, which can also be applied in one go as described in~\cite[Section~3.1]{Young:thesis},
cycles through a finite number of nonisomorphic algebras, because $N(p,r)^{\downarrow}$ is isomorphic to $N(pr,1)$.

Now $\mathcal{N}(q)$, which represents Case~(a) in Theorem~\ref{thm:regular}, has its diamonds in each degree congruent to $1$ modulo $q-1$, all of type $-1$
(except the first, whose type we have agreed to leave undefined).
The Nottingham deflation $\mathcal{N}(q,r)$ has its genuine diamonds,
all of type $-1$,
in all degrees congruent to $q$ modulo $qr-1$.
However, if $r>1$ it also has fake diamonds
between each pair of genuine diamonds,
at locations which make it not regular.
In particular, if $r>1$ then $\mathcal{N}(q,r)$ has a fake diamond of type $1$ in degree $2q-1$.

One may regard $\mathcal{L}_{1,q}$ as a limit case $\mathcal{N}(q,\infty)$.

\subsection{The algebras $\mathcal{T}_{q,1}(M)$ and $\mathcal{T}_{q,2}(M)$ obtained from algebras of maximal class}\label{subsec:T(M)}
All Nottingham algebras of Theorem~\ref{thm:regular}, as well as those of Subsections~\ref{subsec:L(q)} and~\ref{subsec:N(q,r)} have a periodic structure.
This is far from being a universal characteristic of Nottingham algebras, and is rather a rare feature.
In fact, in his PhD thesis~\cite{Young:thesis}
David Young gave two procedures which allow
one to produce two Nottingham algebras
$\mathcal{T}_{q,1}(M)$ and $\mathcal{T}_{q,2}(M)$, both with second diamond $L_q$,
starting from any given  algebra $M$ of maximal class
{\em with at most two distinct two-step centralizers} (see~\cite{CMN,CN} for the latter class, or Section~\ref{sec:construction}).
The diamond patterns of
$\mathcal{T}_{q,1}(M)$
and
$\mathcal{T}_{q,2}(M)$
reflect the pattern of
two-step centralizers of $M$, in two different ways.
Because, over any given field of characteristic $p$,
there are uncountably many
such algebras $M$, and most of them are not periodic,
corresponding assertions carry over to the class of Nottingham algebras, over a given field of characteristic $p$ and for a fixed power $q$ of $p$.

Both algebras
$\mathcal{T}_{q,1}(M)$
and
$\mathcal{T}_{q,2}(M)$
have second diamond $L_q$ (of type $-1$ as always),
and all remaining diamonds are fake or have infinite type.

We start with $\mathcal{T}_{q,2}(M)$,
which is the one relevant to this paper.
Diamonds of infinite type in $L=\mathcal{T}_{q,2}(M)$
occur in sequences of lengths dictated
in a certain way by the structure of $M$,
separated by single occurrences of fake diamonds.
Diamonds of infinite type in these sequences
occur at degree differences of $q-1$.

However, if $L_m$ is a diamond ending any such sequence,
then $L_{m+q-1}$ is a (fake) diamond of type $1$,
and then $L_{m+2q-1}$ begins the next sequence
of diamonds of infinite type.
Thus, the degree difference between $L_{m+q-1}$ and $L_{m+2q-1}$
equals $q$, rather than $q-1$ as in the examples
from Theorem~\ref{thm:regular} which we discussed earlier.
However, if we make use of the intrinsic ambiguity in the definition of fake diamonds,
and view $L_{m+q}$ as a diamond of type $0$, then that has distance
$q-1$ from the next diamond $L_{m+2q-1}$.
In conclusion, the existence of $\mathcal{T}_{q,2}(M)$
does not contradict our general claim
that the distance between two consecutive diamonds
of a Nottingham algebra may always be interpreted
to be equal to $q-1$,
provided that in the presence of fake diamonds
we allow ourselves to choose which component we call the fake diamond,
according to Convention~\ref{convention:fake}.

We stress that, differently from the algebras considered in
Theorem~\ref{thm:regular}, a fake diamond in $\mathcal{T}_{q,2}(M)$
ought to be interpreted
in two different ways (with the corresponding shift by one in degree),
according to which distance we intend to measure
(whether from the previous or to the next diamond).

For the sake of completeness, and in support of quoting Young's
Theorem~\ref{thm:long_second_chain},
we briefly describe the algebras
$\mathcal{T}_{q,1}(M)$.
In contrast with the other class, most diamonds of
$\mathcal{T}_{q,1}(M)$ are fake, interrupted by single occurrences of diamonds
of infinite type (with the necessary exception of $L_q$,
which has type $-1$ as always).
Here, the structure of the underlying algebra of maximal class $M$ dictates,
in a certain way, the lengths of those strings of fake diamonds
between diamonds of infinite type.

Two consecutive fake diamonds in $\mathcal{T}_{q,1}(M)$
occur at degree difference of $q$ if they are interpreted as having
the same type, hence both $1$ or $0$.
Also, the first in such a string of fake diamonds, when read as of type $1$,
occurs at a degree $q-1$ higher
than the preceding genuine diamond
(or, equivalently, degree $q$ higher when read as of type $0$);
analogously,
the last diamond of such a string, when read as of type $0$,
is at a difference degree of $q-1$ from the following genuine diamond.

Thus, we have a situation similar to the one already discussed in Subsection~\ref{subsec:Ell}
for the algebras $\mathcal{L}_{1,q}$:
the general rule of diamonds occurring at a degree distance of $q-1$,
to be satisfied,
requires each fake diamond to be read in two ways,
according to whether one measures the distance from the previous diamond
or to the next.

It may be worth noting that the locations of fake diamonds
in both algebras $\mathcal{T}_{q,1}(M)$
and $\mathcal{T}_{q,2}(M)$ satisfy the following condition on centralizers,
which featured prominently in~\cite{Young:thesis}:
if neither $L_i$ nor $L_{i+1}$ is a genuine diamond, then
\[
C_{L_1}(L_i)=
\begin{cases}
F x&\text{if $d\equiv -1\pmod{q}$}
\\
F y&\text{otherwise},
\end{cases}
\]
where $L_{i-d}$ is the latest genuine diamond preceding $L_i$.

\color{black}

\section{Irregular Nottingham algebras}\label{sec:irregular}
The main goal of this paper is to prove the following theorem,
which constitutes a complete classification of Nottingham algebras.

\begin{theorem}\label{thm:irregular}
Let $L$ be a Nottingham algebra, with second diamond $L_q$.
Then $L$ is isomorphic  
\begin{itemize}
\item[(1)]
to one of the regular algebras described in Cases (a)--(e) of Theorem~\ref{thm:regular}, or
\item[(2)] to one of the following irregular algebras:
\begin{itemize}
\item[(f)]
one of the algebras $\mathcal{L}_{1,q}$, $\mathcal{L}_{0,q}$ of coclass $2$;
\item[(g)]
a Nottingham deflation $\mathcal{N}(q,r)$ for some $p$-power $r>1$;
\item[(h)]
$\mathcal{T}_{q,1}(M)$ or $\mathcal{T}_{q,2}(M)$,
for some  algebra $M$ of maximal class with two distinct two-step centralizers.
\end{itemize}
\end{itemize}
\end{theorem}
A portion of the classification work was attained in David Young's Thesis~\cite{Young:thesis},
and concerned the special case where the next genuine diamond after $L_q$ occurs later than the expected degree $2q-1$.
The terminology in use at the time, which predated a full study of most of the Nottingham algebras of Theorem~\ref{thm:regular}, did not include fake diamonds.
Instead, it spoke of a {\em long second chain,} as in the title of~\cite{Young:thesis},
referring to the sequence of one-dimensional components
comprised between the second diamond $L_q$ and the next genuine diamond.
This expression being now outdated, in Section~\ref{sec:further} we have described in terms of current terminology
all Nottingham algebras discovered by Young.

For reference, Young's classification of {\em thin  algebras with long second chain} reads as follows, where the case labelling matches
corresponding cases in Theorems~\ref{thm:regular} and Theorem~\ref{thm:irregular}, or special cases of them.

\begin{theorem}[Theorem~5.7 in~\cite{Young:thesis}]\label{thm:long_second_chain}
Let $L$ be a Nottingham algebra, with second diamond $L_q$, such that the diamond $L_{2q-1}$ is not genuine.

Then $L$ is isomorphic to one of the following:
\begin{itemize}
\item[($\mathrm{b}'$)]
the algebra from~\cite{CaMa:Nottingham} with $\mu_3=1$, $\mu_4=3$ (in which case the third diamond occurs in degree $3q-2$);
\item[($\mathrm{b}''$)]
the algebra from~\cite{CaMa:Nottingham} with $\mu_3=0$, $\mu_4=1$, $\mu_5=2$ (in which case the third
diamond occurs in weight $4q-3$);
\item[(f)]
one of the algebras $\mathcal{L}_{1,q}$, $\mathcal{L}_{0,q}$ of coclass $2$;
\item[(g)]
a Nottingham deflation $\mathcal{N}(q,r)$ for some $p$-power $r>1$;
\item[($\mathrm{h}'$)]
$\mathcal{T}_{q,1}(M)$ for some algebra $M$ of maximal class with two distinct two-step centralizers.
\end{itemize}
\end{theorem}


In the notation of~\cite[Section~2]{CaMa:Nottingham},
the label $\mu_i$ of Cases~($\mathrm{b}'$) and~($\mathrm{b}''$)
denotes the type of a diamond in degree $(i-1)(q-1)+1$
for a regular algebra.
The irregular algebras $\mathcal{L}_{1,q}$, $\mathcal{N}(q,r)$ and $\mathcal{T}_{q,1}(M)$
may be interpreted as having $\mu_3=1$ and $\mu_4=0$
in the notation of~\cite[Section~2]{CaMa:Nottingham},
while $\mathcal{L}_{0,q}$ would have $\mu_3=0$ and $\mu_4=1$.

After Young's classification Theorem~\ref{thm:long_second_chain}, it remained to discover, and then classify all those Nottingham algebras
where $L_{2q-1}$ is a genuine diamond.
Theorem~\ref{thm:irregular} is a direct consequence of the following theorem.

\begin{theorem}\label{thm:classification}
Let $L$ be a Nottingham algebra, with second diamond $L_q$, such that $L_{2q-1}$ is a genuine diamond.
Then 
\begin{itemize}
\item either $L$ is isomorphic to one of the Nottingham algebras of Theorem~\ref{thm:regular},
\item 
or $L$ is isomorphic to
$\mathcal{T}_{q,2}(M)$ for some   algebra $M$ of maximal class.
\end{itemize}
\end{theorem}

Here is how Theorem~\ref{thm:classification} will follow from
Theorem~\ref{thm:uniqueness}.
If $L_{2q-1}$ is of finite type (necessarily different from $0$ or $1$),
then $L$ falls into Case~(a) or~(b) of Theorem~\ref{thm:regular}.
Hence, we may assume that 
$L_{2q-1}$ is of infinite type.
If every other diamond of the algebra is of infinite type, then $L$
belongs to Case (e) of Theorem~\ref{thm:regular}.

Therefore, we may assume that $L$ has at least one further diamond of finite type
besides $L_q$.
According to Theorem~\ref{thm:distance}, the earliest diamond of finite type
past $L_q$, that is, the earliest diamond of finite type
after a first bunch of diamonds of infinite type,
is $L_m$ with $m\equiv 1\pmod{q-1}$.
Here, $L_m$ could possibly be fake.
According to Theorem~1.2 of~\cite{AviMat:earliest}, for some $s > 0$
either $m = (p^s + 1)(q - 1) + 1$, or $m = 2p^s(q - 1) + 1$ and $L_m$ is fake of type $1$.
In the former case, according to the results of~\cite{AviMat:earliest}, the algebra $L$ is isomorphic to one of the algebras of Case~(c) or~(d) of Theorem~\ref{thm:regular}.
In the latter case, we will show with Theorem~\ref{thm:uniqueness} that
$L$ is isomorphic to $\mathcal{T}_{q,2}(M)$ for some  algebra $M$ of maximal class.

Theorem~\ref{thm:uniqueness} actually refers to a class
$\mathcal{T}_{q,2}$ of algebras, introduced in Definition~\ref{def:Tq2},
that captures the
salient properties of all algebras in the family $\mathcal{T}_{q,2}(M)$.
While David Young provided a construction of
$\mathcal{T}_{q,2}(M)$ in terms of $M$,
for an application of Theorem~\ref{thm:uniqueness}
we need the reverse argument,
namely, that every algebra in class
$\mathcal{T}_{q,2}$
actually arises as $\mathcal{T}_{q,2}(M)$ for a certain $M$.
This is the goal of Section~\ref{sec:construction}.

\newpage
\section{The class $\mathcal{T}_{q,2}$ of Nottingham algebras.}\label{sec:construction}

In his PhD Thesis~\cite{Young:thesis} David Young
constructed a family $\mathcal{T}_{q,2}(M)$ of Nottingham algebras with all diamonds (except the second one) of infinite type or fake. 
Each of these algebras originates from an algebra of maximal class $M$ with
exactly two $2$-step centralizers.
For the reader's convenience, we outline Young's construction in Proposition~\ref{prop:MtoL}.
In the remainder of this section, we characterize, with the present classification in mind, the algebras constructed by Young within the class of Nottingham algebras.

We start with recalling the algebras of maximal class that are the ingredient of Young's construction.

A (graded Lie) algebra $M = \bigoplus_{i=1}^{\infty} M_{i}$ is said to be of maximal class
if $\dim(M_{1}) = 2$, $\dim(M_{i}) \le 1$ for $i > 1$,
and $M_{i} = [M_{i-1} M_{1}]$ for $i > 1$.
As for thin algebras, we restrict our attention to the infinite-dimensional  algebras
of maximal class, that is, from now on we assume $\dim(M_{i}) = 1$ for $i > 1$.
The two-step centralizer $C_{i} = C_{M_{1}}(M_{i})$, for $i > 1$ is then a one-dimensional subspace of $M_{1}$.
The sequence of centralizers $C_{i}$ determines $M$ up to isomorphism. We are interested in two particular cases.
The first case occurs when all centralizers $C_{i}$ coincide.
Choose an element $Y \ne 0$ in $C_{2}$, and an element $X \in M_{1} \setminus C_{2}$.
Set $U_{1} = Y$, and recursively $U_{i} = [U_{i-1} X]$ for $i > 1$.
Then $M_{i}$ is spanned by $U_{i}$, for $i \ge 2$.
This algebra is a metabelian algebra of maximal class, as it is immediate
that the linear span of the $U_{i}$ is an abelian ideal, and it is clearly unique as such.
The other case in which we are interested occurs when the set $\{ C_{i} : i \ge 2 \}$
contains exactly two elements, that is, $M$ has exactly
two (distinct) $2$-step centralizers.
Let $Y$ be a non-zero element in $C_{2}$, and $X$ be a non-zero element
in the other two-step centralizer.
Define $U_{1} = Y$ and then, for $i > 1$ set
\[
U_{i}
=
\begin{cases}
[U_{i-1} X] &\textrm{if $C_{i-1}=\F Y$},\\
[U_{i-1} Y] & \textrm{if $C_{i-1}=\F X$}.
\end{cases}
\]
Then $M_{i}$ is spanned by $U_{i}$, for $i \ge 2$.
We denote by $\mathcal{M}$ the family of  algebras of maximal class with at most two distinct two-step centralizers.
Note that the algebras in $\mathcal{M}$  are precisely those  algebras of maximal class  that are bigraded
in the sense of Section~\ref{sec:grading}, but we will return to this point
after the proof of Proposition~\ref{prop:MtoL}.

Young's construction takes an algebra $M\in\mathcal{M}$, embeds it
into a larger Lie algebra, and identifies a certain Nottingham algebra $\mathcal{T}_{q,2}(M)$ as a subalgebra of the latter.

Let us introduce the following definition.
\begin{definition}\label{def:Tq2}
Let $\mathcal{T}_{q,2}$ denote the class of Nottingham algebras $L$
such that every genuine diamond after the second diamond $L_q$ has type $\infty$,
and whenever $L_{m}$ is a genuine diamond,
the next genuine diamond is either $L_{m+q-1}$ or $L_{m+2q-1}$.
\end{definition}

Note that in the latter case of Definition~\ref{def:Tq2}, that is,
when $L_{m+2q-1}$ is the next genuine diamond after the genuine diamond $L_m$,
according to Theorem~\ref{thm:distance}
there must be precisely one fake diamond in between, and it can only be read as
being $L_{m+q-1}$ of type $1$ or, equivalently, $L_{m+q}$ of type $0$.

Each algebra $\mathcal{T}_{q,2}(M)$ belongs to the class $\mathcal{T}_{q,2}$. Conversely, we will show in Theorem~\ref{thm:one-to-one} that each algebra in $\mathcal{T}_{q,2}$ is of the form $\mathcal{T}_{q,2}(M)$, for a unique $M \in \mathcal{M}$.

Young's construction of $\mathcal{T}_{q,2}(M)$ starts with extending the scalars of $M$ to a divided power algebra,
which we recall now.
Let $\F$ be a field of characteristic $p$, and $q=p^n$ for some $n\geq 1$. Let $\F[\ee;q]$ denote
the divided power algebra with basis $\{\ee^{(i)}:0\leq i\leq q-1\}$
and multiplication $\ee^{(i)}\cdot \ee^{(j)}=\tbinom{i+j}{i}\ee^{(i+j)}$,
where we may write $\ee^{(0)}$ as $1$.
Let $\partial:\F[\ee;q]\rightarrow \F[\ee;q]$ be
the {\em standard} derivation of $\F[\ee;q]$, defined on the basis elements by
$\partial:\ee^{(i)}\mapsto \ee^{(i-1)}$ when $i>0$ and $\partial 1=0$.

\begin{prop}[Theorem~3.3 in \cite{Young:thesis}]\label{prop:MtoL}
Let $M\in \mathcal{M}$ and let $L$  be the subalgebra of $M\otimes \F[\ee;q]+\F\cdot(1\otimes \partial)$ generated by
\[
x=1\otimes \partial \quad \textrm{and} \quad y=X\otimes \ee^{(q-2)}+Y\otimes \ee^{(q-1)}.
\]
Then $L\in \mathcal{T}_{q,2}$.
\end{prop}

Following~\cite{Young:thesis}, we denote the Nottingham algebra $L$ thus constructed by $\mathcal{T}_{q,2}(M)$.

\begin{proof}[Sketch of proof]
For $M\in \mathcal{M}$, with generators $X$ and $Y$ chosen as described earlier,
consider $M\otimes \F[\ee;q]+\F\cdot(1\otimes \partial)$,  and let $L$  be the subalgebra generated by
\[
x=-1\otimes \partial \quad \textrm{and} \quad y=X\otimes \ee^{(q-2)}+Y\otimes \ee^{(q-1)}.
\]

It is easy to check that each homogeneous component $L_{i+1}$ of $L$, for $1\leq i\leq q-2$, is one-dimensional,
spanned by $[y x^{i}]=X\otimes \ee^{(q-2-i)}+Y\otimes \ee^{(q-1-i)}$, and centralized by $\F y$ for $i<q-2$.
In particular, $L_{q-1}$ is spanned by $v_{1}=[y x^{q-2}]=X\otimes 1+Y\otimes \ee^{(1)}$.
The homogeneous component  $L_q$ is a diamond, spanned by $[v_{1}x]=Y\otimes 1$ and
\[
[v_{1}y]
=[X\otimes 1,Y\otimes \ee^{(q-1)}]
+[Y\otimes \ee,X\otimes \ee^{(q-2)}]
=-2 U_{2}\otimes \ee^{(q-1)},
\]
where $U_{2}=[Y X]$ according to notation introduced earlier in this section.
Since  $[v_{1}x x]=0=[v_{1}y y]$ and $[v_{1}y x]=-2U_{2}\otimes \ee^{(q-2)}=-2[v_{1}x y]$, one concludes that $L_q$ is a diamond of type $-1$.

Each one-dimensional homogeneous component $L_{m-1}$ of $L$ spanned by an element of the form $v=U_{i}\otimes \ee^{(1)}$, is just before a diamond of $L$.
Indeed, if $[U_{i}Y]=0$ then $L_{m}$ is spanned by $[v x]=U_{i}\otimes 1$ and $[v y]=-U_{i+1}\otimes \ee^{(q-1)}$.
Furthermore, we have
$[v x x]=0=[v y y]$ and $[v x y]=U_{i+1}\otimes \ee^{(q-2)}=-[v y x]$,
and hence $L_{m}$ is a diamond of infinite type.
If $[U_{i} X]= 0$, then $L_{m}$ is a fake diamond of type $1$,
because $[v x]=U_{i}\otimes 1$ and $[v y]=[v x x]=0$.
Note that in this case $[v x y]=U_{i+1}\otimes \ee^{(q-1)}$, hence the next diamond is $L_{m+q}$.
Thus we are in the second alternative of case (b) of Theorem~\ref{thm:distance}.
The diamond $L_{m+q}$ has necessarily infinite type, because occurrences of $\F X$ are isolated in the sequence of $2$-step centralizers of $M$.

Finally, each homogeneous component of $L$ spanned by an element of the form $U_{i}\otimes \ee^{(j)}$ with $1<j<q$ is centralized by $y$.
The collection of calculations noted so far shows that $L$ is thin, and in fact a
Nottingham algebra with second diamond $L_q$.
Note also that the first diamond of $L$ of type $1$ occurs in degree $2p^{s}(q-1)+1$, for some $s \geq 1$.
Consequently, all Nottingham algebras thus constructed belong to the class $\mathcal{T}_{q,2}$.
\end{proof}

Any algebra $M\in \mathcal{M}$ is also naturally $(\Z\times\Z)$-graded by assigning any two independent bidegrees to $X$ and $Y$.
However, the natural choice in the context of Proposition~\ref{prop:MtoL} is to assign
$X$ and $Y$ bidegrees $(q-2,1)$ and $(q-1,1)$, respectively.
If we extend this double grading of $M$ to a double grading of
$M\otimes \F[\ee;q]$
by assigning bidegree $(-1,0)$ to $\ee^{(1)}$,
then the derivation $x=-1\otimes\partial$ of $M\otimes \F[\ee;q]$
acquires bidegree $(1,0)$.
Furthermore,
$y=X\otimes \ee^{(q-2)}+Y\otimes \ee^{(q-1)}$
is also a bihomogeneous element, of bidegree $(0,1)$.
Thus the Nottingham subalgebra $L$ constructed from $M$ as in
Proposition~\ref{prop:MtoL} inherits precisely the standard double grading of Nottingham algebras defined
in Section~\ref{sec:grading}.

The elements $U_i\otimes 1$ of $L$ span a subalgebra isomorphic with
the maximal ideal of $M$ that contains $Y$, after identifying each $U_i\otimes 1$ with $U_i$.
Thus, in order to recover $M$ from $L$ we only need to produce a derivation of $L$, of bidegree $(q-2,1)$, that will play the role of $X$ in $M$.
In the next result we produce such a derivation, not just of $L=\mathcal{T}_{q,2}(M)$,
but of any algebra $L\in\mathcal{T}_{q,2}$.

In the rest of this section, we establish a natural one-to-one correspondence between the class $\mathcal{T}_{q,2}$ and the class $\mathcal{M}$.
In one direction the correspondence is given by the above construction.
For the opposite direction, we first construct
a suitable derivation of each algebra in $\mathcal{T}_{q,2}$.

\begin{prop}\label{prop:D}
Let $L\in \mathcal{T}_{q,2}$.
\begin{enumerate}
\item There exists a unique derivation  $D:L\rightarrow L$  such that
\[
Dx=0 \quad \textrm{and}\quad Dy=[yx^{q-2}y]=[v_{1}y].
\]
\item
If $L_m$ and $L_{m+q-1}$ are diamonds of $L$,
with $L_{m}$ of infinite type and $L_{m+q-1}$ either of infinite type or fake of type $1$,
then $D[vx]=-2[wx]$, where $0 \neq v \in L_{m-1}$ and $w=[vxyx^{q-3}]\in L_{m+q-2}$.
\item
If $L_m$ and $L_{m+q}$ are diamonds of $L$, with $L_{m}$ fake of type $1$ and $L_{m+q}$ of infinite type, then $D[vx]=0$, where $0 \neq v \in L_{m-1}$.
\end{enumerate}
\end{prop}

\begin{proof}
Informally, $D$ determines a linear map on $L_1$,  which we extend
to a linear map on $\bigoplus_{i=1}^kL_i$ inductively on $k$
by imposing that it satisfies the Leibniz' rule at each stage.
For this to succeed it suffices to make sure that $D$
agrees with each homogeneous relation $w=0$ that holds in $L$,
meaning that $Dw=0$.
While doing this to establish Assertion~(1) we will also prove
Assertions~(2) and ~(3).

In formal terms, if $x$ and $y$ are standard generators for $L$ we view
$L$ as a quotient ${\mathcal F}/R$ of the free Lie algebra ${\mathcal F}$ on $x$ and $y$.
If $D$ is the unique derivation of $F$ that satisfies the conditions in
Assertion~(1), to prove that $D$ actually defines a derivation of $L={\mathcal F}/R$
it suffices to show $DR\subseteq R$.
We do that by proving $DR_k\subseteq R_{k+q-1}$ by induction on $k$,
which may be more conveniently viewed as $DR_k\equiv 0\pmod{R}$.
Now $R_k=\langle w_1\rangle+\langle w_2\rangle+[R_{k-1},x]+[R_{k-1},y]$,
where $w_1=0$ and $w_2=0$ are the at most two independent homogeneous relations
that hold for $L$ in degree $k$.
Since working inductively we may assume
$D[R_{k-1},x]=[DR_{k-1},x]\equiv 0\pmod{R}$
and
$D[R_{k-1},y]=[DR_{k-1},y]+[R_{k-1},Dy]\equiv 0\pmod{R}$,
we see that it suffices to check
$Dw_1\equiv 0\pmod{R}$ and $Dw_2\equiv 0\pmod{R}$.
To avoid constantly writing congruences modulo $R$,
we revert for ease of notation to the more informal writing
$Dw=0$ that we employed in the previous paragraph.

The double grading of $L$ and the support argument introduced in
Section~\ref{sec:grading} will save us several calculations.
Formally, viewing $L={\mathcal F}/R$ as above, the ideal of relations $R$
is a graded ideal with respect to the $(\Z\times\Z)$-grading of ${\mathcal F}$
obtained by declaring $x$ and $y$ to have degrees $(1,0)$ and $(0,1)$.
When the result of a calculation on homogeneous elements of ${\mathcal F}$
lies in a bidegree that falls outside the support of $L$, we automatically
know that the result is zero modulo $R$.
In this respect, note that
$D{\mathcal F}_{(r,s)}\subseteq{\mathcal F}_{(r+q-2,s+1)}$.

One instance of this is when $[L_{k-1},y]=0$.
This occurs, in particular, when $2<k<q$.
More generally, it occurs whenever $L_{k-1}$ is a homogeneous component of $L$
which is not a diamond or immediately precedes a diamond,
with an obvious interpretation in case that diamond is fake.
Then $L_{k-1}=\langle u\rangle=L_{(r,s)}$, say, and in degree $k$ we have
the single relation $[uy]=0$.
Now $D[uy]\in L_{(r+q-2,s+2)}$, but
according to Theorem~\ref{thm:distance}
we have $[L_{k+q-2},y]=0$, and also $L_{k+q-2}=L_{(r+q-2,s+1)}$.
Hence $L_{(r+q-2,s+2)}=0$ and so we conclude $D[uy]=0$
without actually performing the (easy) calculation.
In short, we draw the conclusion from the fact that the bidegree $(r+q-2,s+2)$
lies outside the support of $L$ according to Theorem~\ref{thm:distance}.

Next, there is no relation to check in degree $q$.
More generally, there is no relation to check in degree $k$ whenever
$\dim(L_{k-1})=1$ and $\dim(L_k)=2$, that is, when $L_k$ is a genuine diamond
(including the case of $L_1$).

To determine the effect of $D$ so far, note that because $Dx=0$
the derivation $D$ commutes with the inner derivation $\ad x$.
Hence $D[ux^i]=[Du,x^i]$ for $u\in L$, for every positive integer $i$.
In particular,  we have
$D[yx^i]=[v_{1}yx^i]$ and $Dv_{1}=[v_{1}yx^{q-2}]=-2v_{2}$.

In degree $q+1$ we have the relations $[v_{1}xx]=0$, $[v_{1}yy]=0$ and
$[v_{1}yx]+2[v_{1}xy]=0$.
Then in degree $2q$ we have $D[v_{1}xx]=0$ and $D[v_{1}yy]=0$ by the support argument.
In fact, according to Theorem~\ref{thm:distance} the bidegrees of the left-hand
sides equal the formal bidegrees of $[v_{2}xx]$ and $[v_{2}yy]$,
and those bidegrees lie outside the support of $L$.
As to the third diamond relation, we find $D[v_{1}yx]=-4[v_{2}yx]$, and
\begin{align*}
D[v_{1}xy]&=[Dv_{1},xy]+[v_{1}x,Dy]
=-2[v_{2}xy]+[v_{1}x[v_{1}y]]=2[v_{2}yx],
\end{align*}
where
$[v_1x[v_1y]]=[v_1xv_1y]-[v_1xyv_1]=2[v_2xy]+2[v_2yx]$
according to Lemma~\ref{lemma:v1}.
Hence we obtain $D([v_{1}yx]+2[v_{1}xy])=0$ as desired.
So far we have proved that $D$ agrees with all homogeneous relations in $L$
up to degree $q+1$.

Now we proceed inductively, in batches from one diamond to the next.
Thus, we assume Assertions~(1)--(3) hold for all diamonds up to and including $L_m$,
and prove that they hold for the next diamond $L_{m+q-1}$ as well,
whether of infinite type or fake of type $1$.
The inductive step splits into two cases, according to the inductive hypothesis being provided by
Assertion~(2), or~(3).

First suppose $L_m$ and $L_{m+q-1}$ are diamonds,
with $L_m$ of infinite type and $L_{m+q-1}$ of infinite type or fake of type $1$.
Let $0\neq v\in L_{m-1}$ and set $w=[vxyx^{q-3}]\in L_{m+q-2}$.
According to Assertion~(2) we may inductively assume $Dv=-2w$.
Then we have $D[vx]=[Dv,x]=-2[wx]$, and
\[
D[vy]=[Dv,y]+[v,Dy]=-2[wy]+[v[v_1y]]=-2[wy],
\]
where $[v[v_1y]]=[vv_1y]-[vyv_1]=0$ according to Lemma~\ref{lemma:v1}
since $L_m$ has infinite type.
Proceeding one degree higher, the support argument gives us $D[vxx]=0$ and $D[vyy]=0$ for free.
Furthermore, $D[vyx]=-2[wyx]$ and
\[
D[vxy]=[Dv,xy]+[vx,Dy]=-2[wxy]+[vx[v_{1}y]]=2[wyx],
\]
where
$[v_1x[vy]]=[vxvy]-[vxyv]=2[wxy]+2[wyx]$
according to Lemma~\ref{lemma:v1}.
Hence we obtain $D([vxy]+[vyx])=0$ as desired.
Next, as we argued earlier the support argument yields $D[vxyx^iy]=0$ for $0\le i<q-3$.
Finally,
\[
Dw=D[vxyx^{q-3}]=[D[vxy],x^{q-3}]=2[wyx^{q-2}]
\]
Finally, note that $Dw=D[vxyx^{q-3}]=2[wyx^{q-2}]$,
Hence
$
0\neq Dw=-2[wxyx^{q-3}]\in L_{m+2q-3}
$
if $L_{m+q-1}$ has infinite type,
and $Dw=0$ if $L_{m+q-1}$ is fake of type $1$.

Now suppose $L_m$ and $L_{m+q}$ are consecutive diamonds,
with $L_{m}$ fake of type $1$ and $L_{m+q}$ of infinite type.
Let $0\neq v\in L_{m-1}$ and set $w=[vxyx^{q-2}]\in L_{m+q-1}$.
According to Assertion~(3) we may inductively assume $Dv=0$.
Then we have $D[vx]=[Dv,x]=0$, and $D[vy]=0$ by the support argument.
One degree higher we have $D[vxx]=[Dv,xx]=0$, and
\[
D[vxy]=[vx,Dy]=[vx[v_{1}y]]=[vxv_{1}y]-[vxyv_{1}]=[wy]+[wy]=2[wy],
\]
where we have used Lemma~\ref{le:type_1}.
Consequently, we have
\[
D[vxyy]=2[wyy]+[vxy[v_{1}y]]=[vxyv_{1}y]=-[wyy]=0
\]
Next, the support argument yields $D[vxyx^iy]=0$ for $0<i\le q-3$.
Finally, note that $0\neq Dw=D[vxyx^{q-2}]=-2[wxyx^{q-3}]\in L_{m+2q-2}$.
Taking into account that $Dv_{1}=-2v_{2}$, an inductive argument completes the proof.
\end{proof}

If $L\in \mathcal{T}_{q,2}$ is given a $(\Z\times\Z)$-grading as in
Section~\ref{sec:grading}, so that $x$ and $y$  have degrees $(1,0)$ and $(0,1)$, then
the derivation $D$ constructed in Proposition~\ref{prop:D}
is a graded derivation of $L$, of bidegree $(q-1,1)$.
That derivation will now be instrumental in producing an algebra of maximal class in $\mathcal{M}$
starting from any Nottingham algebra in the class $\mathcal{T}_{q,2}$.
This direction of the correspondence was not made explicit in~\cite{Young:thesis}.

\begin{prop}\label{prop:LtoM}
 Let $L\in \mathcal{T}_{q,2}$ and let $D\in \Der(L)$ be the derivation defined in
 Proposition~\ref{prop:D}.
 Let $M$ be the subalgebra of $L+\F D$ generated by $X=D$ and $Y=[yx^{q-1}]$. Then $M\in \mathcal{M}$.
\end{prop}

\begin{proof}
Let $x$ and $y$ be standard generators of $L$.
According to  Proposition~\ref{prop:D}, there exists a derivation $D\in \Der(L)$ such that
$Dx=0$ and $Dy=[v_{1}y]=[yx^{q-2}y]$.

Set $U_{1}=Y$ and for $j> 0$ define recursively
\[
U_{j+1}=\begin{cases}
[U_{j}X] &\textrm{if $[U_{j}Y]=0$,}\\
[U_{j}Y] &\textrm{otherwise}.
\end{cases}
\]
In particular, $U_{2}=[YX]=-2[v_{2}x]$ and $[YXY]=-2[v_{2}x[v_{1}x]]=0$.

We now show that $M$ is an algebra of maximal class with at most two $2$-step centralizers.

Suppose by induction that $U_{j}$ lies in a diamond $L_{m}$ of infinite type of $L$, for some $j\geq 2$. Then $L_{m+q-1}$ is the next diamond of $L$, either of infinite type or fake of type $1$. We can also assume $U_{j}=[vx]$ for some $0\neq v \in L_{m-1}$,
up to a nonzero multiplicative coefficient. Since $[U_{j}Y]=[vx[v_{1}x]]=[vxyx^{q-1}]=0$, we have $U_{j+1}=[U_{j}X]=-2[wx]$, where $w=[vxyx^{q-3}]$ spans $L_{m+q-2}$.

On the other hand, if $U_{j}$ lies in a fake diamond $L_{m}$ of type $1$, then $L_{m+q}$ is the next diamond in $L$, of infinite type. We can assume $U_{j}=[vx]$, for $0\neq v \in L_{m-1}$. We have $[U_{j}X]=0$, hence $U_{j+1}=[U_{j}Y]=[wx]$, where $w=[vxyx^{q-2}]$ spans $L_{m+q-1}$.

We conclude that $M\in \mathcal{M}$, as desired.
\end{proof}

We are now ready to establish the announced goal of this section.

\begin{theorem}\label{thm:one-to-one}
Let $L\in \mathcal{T}_{q,2}$, then $L$ is isomorphic to
an algebra $\mathcal{T}_{q,2}(M)$,
for a unique $M\in\mathcal{M}$.
\end{theorem}

\begin{proof}
Let $x$ e $y$ be standard generators of $L$, and let $X=D$ and $Y=[yx^{q-1}]$ be as in Proposition~\ref{prop:LtoM}. According to Proposition~\ref{prop:LtoM}, the algebra $M$ generated by $X$ and $Y$ is an algebra of maximal class, with exactly two $2$-step centralizers.

Consider $M\otimes \F[\ee;q]+\F\cdot(1\otimes \partial)$,  and let $T$  be the subalgebra generated by
\[
\Tilde{x} =-1\otimes \partial \quad \textrm{and} \quad \Tilde{y}=X\otimes \ee^{(q-2)}+Y\otimes \ee^{(q-1)}=D\otimes \ee^{(q-2)}+[v_{1}x]\otimes \ee^{(q-1)}.
\]
We already know that $T\in \mathcal{T}_{q,2}$.
Setting  $\tilde v_{1}=[\tilde y \tilde x^{q-2}]=X\otimes 1+Y\otimes \ee^{(1)}$,
the homogeneous component  $T_q$ is  spanned by $[\Tilde{v}_{1}\Tilde{x}]=Y\otimes 1=[v_1 x]\otimes 1$ and
$
[\Tilde{v}_{1}\Tilde{y}]
=-2 U_{2}\otimes \ee^{(q-1)}=4[v_{2}x]\otimes \ee^{(q-1)}.
$

Each one-dimensional homogeneous component $T_{n-1}$ of $T$ that is spanned
by an element of the form $\tilde v=U_{i}\otimes \ee^{(1)}$, is just before a diamond of $T$.
Furthermore, $U_{i}$ lies in a diamond $L_{m}$ of $L$ and we can assume $U_{i}=[vx]$,
for some $0\neq v\in L_{m-1}$.

If $L_{m}$ has infinite type, then $[U_{i}Y]=0$ and  $T_{n}$ is a diamond of infinite type,
spanned by $[\tilde v \tilde x]=U_{i}\otimes 1=[vx]\otimes 1$ and
$[\tilde v  \tilde y]=-U_{i+1}\otimes \ee^{(q-1)}=2[wx]\otimes \ee^{(q-1)}$,
where $w=[vxyx^{q-3}]$.
If $L_{m}$ is a fake diamond of type $1$, then  $[U_{i}X]= 0$ and  $T_{n}$
is a fake diamond of type $1$, since
$[\tilde v \tilde x]=U_{i}\otimes 1=[vx]\otimes 1$ and
$[\tilde v \tilde y]=[\tilde v  \tilde x \tilde x]=0$.
\end{proof}

\section{A membership criterion for $\mathcal{T}_{q,2}$}\label{sec:mixed}

The aim of this section is to prove that if the structure of a Nottingham algebra $L$
matches that of an algebra in $\mathcal{T}_{q,2}$
up to the first diamond of type $1$, then $L\in\mathcal{T}_{q,2}$.
This is expressed rigorously in the following result.

\begin{theorem}\label{thm:uniqueness}
Let $L$ be a Nottingham algebra with second diamond $L_{q}$.
Suppose $L$ has diamonds of infinite type in each degree $k(q-1)+1$
with $2\leq k <2p^s$,
and a fake diamond of type $1$ in degree $2p^s(q-1)+1$, for some $s\geq 1$.
Then each diamond $L_{m}$ of $L$ is either of infinite type
or fake of type $1$.
In the latter case we also have $[L_{m+q-2}y]=0$.
In particular, we have $L\in \mathcal{T}_{q,2}$.
\end{theorem}

We divide the proof into several steps.
To start the ball rolling, an argument in~\cite[Subsection~5.2]{AviMat:earliest}
guarantees that, if a diamond of type $1$ of a Nottingham algebra
is preceded by a diamond of infinite type, then it is followed by a diamond of infinite type,
at a degree difference of either $q-1$ or $q$.
For the reader's convenience we state this fact as a lemma
and quote its proof in a more formal way.

\begin{lemma}[Subsection~5.2 of~\cite{AviMat:earliest}]\label{le:one_past_L_m}
Let $L$ be a Nottingham algebra, with $L_{2q-1}$ a diamond of infinite type.
Let $L_{m}$ be a diamond of $L$ of type $1$, for some $m$,
with $L_{m-q+1}$ a diamond of infinite type.
Then either
$L_{m+q-1}$ or $L_{m+q}$ is a diamond of infinite type.
\end{lemma}

\begin{proof}
According to Theorem~\ref{thm:distance}, either $L_{m+q-1}$ or $L_{m+q}$ is a diamond.
Let $u$ be a nonzero element of $L_{m-q}$, set
$v=[uxyx^{q-3}]\in L_{m-1}$, and $w=[vxyx^{q-3}]\in L_{m+q-2}$.
Expanding
\[
[ux[v_{2}xy]]+[ux[v_{2}yx]]=0
\]
we obtain
\[
[uxv_{2}xy]-2[uxyv_{2}x]+[uxyxv_{2}]+[uxv_{2}yx]=0.
\]
By means of Lemma~\ref{lemma:v2ext} with $\mu=1$ we deduce $[wxxy]+[wxyx]+[wyxx]=0$.
If $[wy]\neq 0$, then because $[wxx]=0$ we have $[wxy]+[wyx]=0$,
and hence $L_{m+q-1}$ is a diamond of infinite type.
However, if $[wy]=0$, then $L_{m+q}$ is a diamond of infinite type,
because $[zxy]+[zyx]=0$ for  $z=[wx]$.
\end{proof}

According to Lemma~\ref{le:one_past_L_m},
the degree difference between the last diamond
preceding the fake diamond $L_m$, and the first diamond following it
(both of which have infinite type),
equals either $2q-2$ or $2q-1$.
As it turns out, under the hypotheses of Theorem~\ref{thm:uniqueness},
 the first possibility cannot occur
if $L_m$ is the first diamond of type $1$ in order of occurrence,
that is, when $m=2p^s(q-1)+1$.
This is the content of the next proposition.
Before doing that we introduce another piece of notation to describe the structure of $L$
up to degree $2p^s(q-1)+1$.
\begin{definition}
We recursively define
\[
v_{k}=[v_{k-1}xyx^{q-3}],\quad \textrm{for $2<k\leq 2p^{s}$}.
\]
\end{definition}
In particular, for $1<k<2p^s$ the element $v_{k}$ spans the homogeneous component just before a diamond of infinite type, thus $[v_{k}yx]+[v_{k}xy]=0$, while $[v_{2p^s}y]=0$.
\begin{prop}\label{prop:dichotomy_first}
If $L$ is as in Theorem~\ref{thm:uniqueness} and $m=2p^s(q-1)+1$,
then $[L_{m+q-2}y]=0$.
\end{prop}
\begin{proof}
Assume by way of contradiction $[L_{m+q-2}y]\neq 0$, whence $L_{m+q-1}$ is a diamond, necessarily of infinite type.
Set $v_{2p^{s}+1}=[v_{2p^s}xyx^{q-3}]$, thus $[v_{2p^{s}+1}yx]$ spans the homogeneous component $L_{m+q}$.
We will show that $[v_{2p^{s}+1}yx]$ vanishes, against the hypothesis that $L$ has infinite dimension.

Consider the element $u=[v_{p^s}xyx^{(q-3)/2}]$. We expand
\begin{align*}
 0&=\pm  [uu] =[u[v_{p^s}xyx^{(q-3)/2}]]=
 \sum_{i=0}^{(q-3)/2}(-1)^{i}\binom{(q-3)/2}{i}[ux^{i}[v_{p^s}xy]x^{(q-3)/2 -i}]\\
 &=\sum_{i=0}^{(q-3)/2}(-1)^{i}\binom{(q-3)/2}{i}[ux^{i}[v_{p^s}x]yx^{(q-3)/2 -i}]\\
 &\quad-\sum_{i=0}^{(q-3)/2}(-1)^{i}\binom{(q-3)/2}{i}[ux^{i}y[v_{p^s}x]x^{(q-3)/2 -i}].
\end{align*}
All terms of the former sum vanish except, possibly, for $i=(q-5)/2$ and $i=(q-3)/2$.
All terms of the latter sum vanish except, possibly, for $i=(q-3)/2$. Consequently, we find
\begin{equation}\label{eq:dichothomy_first}
\begin{split}
 0=&-\frac{3}{2}[v_{p^{s}+1}^{-1}v_{p^s}xyx]+\frac{3}{2}[v_{p^{s}+1}v_{p^s}yx]-[v_{p^{s}+1}v_{p^s}xy]
 +[v_{p^{s}+1}xv_{p^s}y]\\
 &+[v_{p^{s}+1}yv_{p^s}x]+[v_{p^{s}+1}xyv_{p^s}],
\end{split}
\end{equation}
where $v_{p^{s}+1}^{-1}=[v_{p^s}xyx^{q-4}]$ according to Notation~\eqref{notation:v^-1}.

We expand the Lie products in the above equation by means of
Lemmas~\ref{lemma:v1},~\ref{lemma:v2} and~\ref{lemma:v2ext}.
We start with
\[
[v_{p^{s}+1}v_{p^s}]=[v_{p^{s}+1}[v_{2}v_{1}^{p^{s}-2}]]=\sum_{i=0}^{p^{s}-2}
(-1)^i \binom{p^{s}-2}{i}[v_{p^{s}+1}v_{1}^iv_2v_{1}^{p^{s}-2-i}].
\]
The Lie product in the sum vanishes, except possibly when $i=p^{s}-2$, whence
\[
[v_{p^{s}+1}v_{p^s}]=-[v_{2p^{s}-1}v_{2}]=2v_{2 p^{s}+1}.
\]
Similarly, we have $[v_{p^{s}+1}xv_{p^s}]=[v_{2p^{s}+1}x]$ and
$[v_{p^{s}+1}yv_{p^s}]=[v_{2p^{s}+1}y]$.

We now compute
\[
[v_{p^{s}+1}xyv_{p^s}]=\sum_{i=0}^{p^{s}-2}(-1)^i\binom{p^{s}-2}{i}
[v_{p^{s}+1}xyv_{1}^{i}v_2 v_{1}^{p^{s}-2-i}].
\]
The Lie product in the sum vanishes, except possibly when $i=p^s-3$ and $i=p^s-2$. Hence, we obtain
\begin{align*}
[v_{p^{s}+1}xyv_{p^s}]
&=-2[v_{p^{s}+1}xyv_{1}^{p^{s}-3}v_2v_1]
-[v_{p^{s}+1}xyv_{1}^{p^{s}-2}v_2]
\\&=4[v_{2p^{s}+1}yx]+2[v_{2p^{s}+1}xy]
-[v_{2p^{s}+1}xy]-3[v_{2p^{s}+1}yx]
\\&=[v_{2p^{s}+1}yx]+[v_{2p^{s}+1}xy].
\end{align*}
Finally, we have
\[
[v_{p^{s}+1}^{-1}v_{p^s}]=-[v_{p^{s}+1}^{-1}v_{1}^{p^{s}-2}v_2]
=-[v_{2p^{s}-1}^{-1}v_2]=3v_{2p^{s}+1}^{-1}.
\]
Substituting in Equation~\eqref{eq:dichothomy_first},
we get $0=(1/2)[v_{2p^{s}+1}yx]$, as desired.
\end{proof}

\begin{rem}
\label{rem:structure}
It follows that the diamond structure of $L$ is uniquely determined up to degree $(2p^{s}+1)(q-1)+3$.
Each homogeneous component in degree $k(q-1)+1$ with $1\leq k \leq 2p^{s}$ is a diamond.
Each of these diamonds, except for $L_{q}$ and $L_{m}$ with $m=2p^s(q-1)+1$, has infinite type.
The second diamond $L_{q}$ has type $-1$, while $L_{m}$ is a fake diamond of type $1$ with $[L_{m+q-2}y]=0$.
Furthermore, the diamond $L_{(2p^{s}+1)(q-1)+2}$ has infinite type
because of Lemma~\ref{le:one_past_L_m}.
\end{rem}

These facts will form the base step of an inductive argument that proves
Theorem~\ref{thm:uniqueness}.
We will prove the inductive step in the following arguments.

Let  $L_{m}$ be an arbitrary diamond of type $1$ with $[L_{m+q-2}y]=0$,
for some $m\geq 2p^{s}(q-1)+1$.
Suppose that, for a certain $1<r<2p^s$, the homogeneous components $L_{m+q}, L_{m+2q-1}, \ldots, L_{m+(r-1)(q-1)+1}$ are diamonds of infinite type.
Theorem~\ref{thm:distance} implies that $L_{m+r(q-1)+1}$ is a diamond.
In the next proposition we prove that it has either infinite type or type $1$.
Also, in the latter case, $L_{m+(r+1)(q-1)+1}$ cannot be a genuine diamond.

\begin{prop}\label{prop:diamond_r}
Let $L$ as in Theorem~\ref{thm:uniqueness} and  $L_m$ be a diamond of type
$1$ for some $m\geq 2p^s(q-1)+1$, with $[L_{m+q-2}y]=0$. Let $2 \leq r<2p^s$ be an integer.
Suppose that the homogeneous components
\[
L_{m+q}, L_{m+2q-1}, \ldots , L_{m+(r-1)(q-1)+1}
\]
are diamonds of infinite type.

Then in degree $n = m+r(q-1)+1$ there is a diamond either  of infinite type or of type $1$.
In the latter case, we have $[L_{n+q-2}y]=0$.
\end{prop}
\begin{proof}
Assume $L_{m+r(q-1)+1}$ is a diamond of type $\mu\in \F \cup \{\infty\}$.
Theorem~\ref{thm:distance} implies $[L_{m+r(q-1)+q-2}y]=0$.
Define recursively
\begin{equation}\label{eq:v_b+r_def}
v_{b+k}=[v_{b+k-1}xyx^{q-3}],
\end{equation}
for $1<k\leq r$, and
\[
v_{b+r+1}=\begin{cases}
[v_{b+r}xyx^{q-3}] &\textrm{if $\mu \neq 0$}\\
[v_{b+r}yx^{q-2}] &\textrm{otherwise}.
\end{cases}
\]
Thus each element $v_{b+k}$, $0\leq k \leq r$, spans the homogeneous component just preceding a diamond.
We expand the equation
\begin{align}\label{eq:v_r+1}
0&=[v_{b}[v_{r+1}yx]]+[v_{b}[v_{r+1}xy]]\\
&=[v_{b}v_{r+1}xy]-2[v_{b}xv_{r+1}y]+[v_{b}xyv_{r+1}],\nonumber
\end{align}
where we have used $[v_{b}v_{r+1}y]\in [L_{m+r(q-1)+q-2}y]=0$.
Assume $\mu\neq 0$ first. To compute the first summand of Equation~\eqref{eq:v_r+1}, we expand
\begin{align*}
[v_{b}v_{r+1}]&=\sum_{i=0}^{r-1}(-1)^i\binom{r-1}{i}[v_{b}v_{1}^{i}v_{2}v_{1}^{r-1-i}]\\
&=[v_{b}v_{2}v_{1}^{r-1}]+[v_{b}v_{1}^{r-1}v_{2}]=2[v_{b+2}^{-1}v_{1}^{r-1}]+2(-1)^{r-1}[v_{b+r-1}^{-1}v_{2}]\\
&=2(1+2\mu^{-1}+3(-1)^{r}\mu^{-1})v_{b+r+1}^{-1}.
\end{align*}
Similarly, we have
\begin{align*}
[v_{b}xv_{r+1}]&=[v_{b}xv_{2}v_{1}^{r-1}]+(-1)^{r-1}
[v_{b}xv_{1}^{r-1}v_{2}]
\\&=
[v_{b+2}v_{1}^{r-1}]+(-1)^{r-1}[v_{b+r-1}v_{2}]
\\&=
(1+\mu^{-1}+2(-1)^r \mu^{-1})v_{b+r+1},
\end{align*}
and
\[
[v_{b}xyv_{r+1}]=(-1)^{r-1}[v_{b}xyv_{1}^{r-1}v_2]=(-1)^{r}[v_{b+r-1}yv_{2}]=(-1)^{r+1}\mu^{-1}[v_{b+r+1}y].
\]
Substituting in Equation~\eqref{eq:v_r+1} we get
\[
(2+(-1)^{r})\mu^{-1}[v_{b+r+1}y]=0,
\]
whence $\mu=1$ and $[L_{m+(r+1)(q-1)}y]=0$, as desired.

When $\mu=0$, we have
\begin{align*}
[v_{b}v_{r+1}]&=(4+6(-1)^r)v_{b+r+1}^{-1},\\
[v_{b}xv_{r+1}]&=(1+2(-1)^r)v_{b+r+1},\\
[v_{b}xyv_{r+1}]&=(-1)^{r+1}[v_{b+r+1}y].
\end{align*}
Substituting in Equation~\eqref{eq:v_r+1} we get $(2+(-1)^r)[v_{b+r+1}y]=0$, a contradiction.
\end{proof}

We now prove that the difference in degrees of any two diamonds of type $1$
in $L$ is at most $2p^s(q-1)+1$, provided that the first of those two
diamonds is preceded by a diamond of infinite type.

\begin{prop}\label{prop:diamond_1}
Let $L$ as in Theorem~\ref{thm:uniqueness} and  $L_m$ be a diamond of type
$1$ for some $m\geq 2p^s(q-1)+1$, with $[L_{m+q-2}y]=0$.
Suppose that $L_{m-q+1}$ is a diamond of infinite type.
Then a diamond of type $1$ occurs in degree $n\leq m+2p^s(q-1)+1$.
Furthermore, $[L_{n+q-2}y]=0$.
\end{prop}
\begin{proof}
According to Lemma~\ref{le:one_past_L_m} the diamond $L_{m+q}$ has infinite type.
Therefore, Proposition~\ref{prop:diamond_r} inductively shows
that each homogeneous component $L_{m+r(q-1)+1}$
is either a diamond of infinite type or a diamond of type $1$, for $1<r<2p^s$.
In the latter case we have $[L_{m+(r+1)(q-1)}y]=0$.

Suppose $L_{m+r(q-1)+1}$ has infinite type, for every $1\leq r <2p^s$. We expand
the equation $0=[v_{b}x[v_{2p^s}y]]$, obtaining $0=[v_{b}xv_{2p^s}y]-[v_{b}xyv_{2p^s}]$.
To compute the first term of the difference, we expand
\[
[v_{b}xv_{2p^s}]=[v_{b}x[v_{2}v_{1}^{2p^{s}-2}]]=[v_{b}xv_{2}v_{1}^{2p^{s}-2}]+[v_{b}xv_{1}^{2p^{s}-2}v_{2}]=v_{b+2p^s},
\]
where we used the results in Lemmas~\ref{lemma:v1}, \ref{le:type_1} and~\ref{lemma:v2}.
To compute the second term of the difference, we expand
\[
[v_{b}xyv_{2p^s}]=[v_{b}xy[v_{2}v_{1}^{2p^{s}-2}]]=[v_{b}xyv_{1}^{2p^{s}-2}v_{2}]=0.
\]
We deduce $0=[v_{b+2p^s}y]$, whence $L_{m+2p^{s}(q-1)+1}$ has type $1$. To complete the proof, we show that $[L_{m+(2p^s+1)(q-1)}y]=0$.
Set $v_{b+2p^s+1}=[v_{b+2p^s}xyx^{q-3}]\in L_{m+(2p^{s}+1)(q-1)}$. Lemma~\ref{le:type_1} implies that $2v_{2p^s+1}^{-1}=[v_{2}v_{1}^{2p^2-1}]$. We expand
\[
0=2[v_{b}x[v_{2p^s+1}^{-1}y]]=
[v_{b}x[v_{2}v_{1}^{2p^{s}-1}]y]-[v_{b}xy[v_{2}v_{1}^{2p^{s}-1}]].
\]
To compute the first term in the difference, we expand
\[
[v_{b}x[v_{2}v_{1}^{2p^{s}-1}]y]=[v_{b}xv_{2}v_{1}^{2p^s-1}]-[v_{b}xv_{1}^{2p^s-1}v_{2}]=4v_{b+2p^s+1},
\]
by means of Lemmas~\ref{le:type_1}, \ref{lemma:v1}, \ref{lemma:v2} and~\ref{lemma:v2ext}.
Similarly, we have
\[
[v_{b}xy[v_{2}v_{1}^{2p^{s}-1}]]=-[v_{b}xyv_{1}^{2p^s-1}v_{2}]=-[v_{b+2p^s+1}y].
\]
Substituting in the equation above, we deduce $5[v_{b+2p^s+1}y]=0$, as desired.
\end{proof}

We conclude this section by proving Theorem~\ref{thm:uniqueness}.

\begin{proof}[Proof of Theorem~\ref{thm:uniqueness}]
Let $L$ as in the statement. According to Remark~\ref{rem:structure}
the diamond structure of $L$ is determined up to degree $(2p^{s}+1)(q-1)+3$.
Let now $L_{m}$ be a diamond of type $1$ of $L$ such that $[L_{m+q-2}y]=0$ and $L_{m-q+1}$ is a diamond of infinite type, for some $m\geq 2p^{s}(q-1)+1$.
Lemma~\ref{le:one_past_L_m} implies that $L_{m+q}$ is a diamond of infinite type.
Therefore, by Proposition~\ref{prop:diamond_r} and Proposition~\ref{prop:diamond_1},
each homogeneous component $L_{m+r(q-1)+1}$, with $1<r\leq 2p^s$, is either a diamond of infinite type or a  diamond of type $1$.
In the latter case, set $n=m+r(q-1)+1$, then $[L_{n+q-2}y]=0$. Finally, at least a diamond of type $1$ occurs in degree at most  $m+2p^s(q-1)+1$.
\end{proof}

\section{General calculations in a Nottingham algebra}\label{sec:gc}

We briefly recall some results on arbitrary Nottingham algebras,
quoting them from~\cite{AviMat:earliest} and referring the reader to that paper for proofs.

Let $L$ be an arbitrary Nottingham algebra with standard generators $x$ and $y$.
Let $L_m$ be a diamond of $L$, with arbitrary type $\mu\in\F\cup\{\infty\}$, for $m\geq q$ and $0\neq v_k\in L_{m-1}$. Assume $L_{m+q-1}$ is a diamond
(which is a consequence of
Theorem~\ref{thm:distance} only when $L_m$ is genuine).
Then according to
Theorem~\ref{thm:distance}
the element $y$ centralizes $L_{m+1},\ldots,L_{m+q-3}$
Define the element $v_{k+1}$, in degree $m+q-2$, as
\[
v_{k+1}=\begin{cases}
[v_{k}xyx^{q-3}] &\mbox{if } \mu \neq 0, \\
[v_{k}yx^{q-2}] &\mbox{otherwise}.
\end{cases}
\]
We start with describing the adjoint action of $v_{1}$ on the elements close to this diamond. We use the convention $\infty^{-1}=0$.

\begin{notation}\label{notation:v^-1}
When  $v_{h+1}=[v_{h}xyx^{q-3}]$, we write $v_{h+1}^{-1}$ for the element $[v_{h}xyx^{q-4}]$.
\end{notation}
\begin{lemma}\label{lemma:v1}
Suppose $[v_{k}yx]=(\mu^{-1}-1)[v_{k}xy]$ with $\mu \in \F^{\ast}\cup\{\infty\}$.
Then we have
\begin{align*}
&[v_{k}^{-1}v_{1}]=(2 \mu^{-1}+1)v_{k+1}^{-1},\\
&[v_{k}v_{1}]=(\mu^{-1}+1)v_{k+1}, \\
&[v_{k}xv_{1}]=[v_{k+1}x], \\
&[v_{k}yv_{1}]=(1-\mu^{-1})[v_{k+1}y], \\
&[v_{k}xyv_{1}]=-(2[v_{k+1}yx]+[v_{k+1}xy]), \\
&[v_{k}xyxv_{1}]=-(3[v_{k+1}yx^2]+2[v_{k+1}xyx]).
\end{align*}
\end{lemma}
\begin{proof}
All the stated equations, except the first one, follow
from~\cite[Lemma 4.1]{AviMat:earliest}.
Direct calculation shows
\[
[v_{k}^{-1}v_{1}]=[v_{k}^{-1}[yx^{q-2}]]=2[v_{k}yx^{q-3}]+3[v_{k}xyx^{q-4}]=
(2\mu^{-1}+1)v_{k+1}^{-1},
\]
as desired.
\end{proof}
\begin{rem}\label{rem:v1_mu_0}
In the case $\mu=0$ excluded from Lemma~\ref{lemma:v1}, similar calculations show
$[v_{k}^{-1}v_{1}]=2v_{k+1}^{-1}$.
\end{rem}

Assume the diamond $L_{m+q-1}$ has infinite type and set $v_{k+2}=[v_{k+1}xyx^{q-3}]$.
According to Theorem~\ref{thm:distance},
the element $y$ centralizes $L_{m+q},\ldots,L_{m+2q-4}$,
and $L_{m+2q-2}$ is a diamond.
We describe the adjoint action of $v_{2}$ on elements close to the diamond $L_m$.

\begin{lemma}[Lemma~4.5 in~\cite{AviMat:earliest}]\label{lemma:v2}
Suppose $[v_{k}yx]=(\mu^{-1}-1)[v_{k}xy]$, with $\mu\in \F^{\ast}\cup\{\infty\}$, and $[v_{k+1}yx]=-[v_{k+1}xy]$.
Then we have
\begin{align*}
&[v_{k}v_{2}]=\mu^{-1}v_{k+2}, \quad [v_{k}xv_{2}]=0=[v_{k}yv_{2}], \\
&[v_{k}xyv_{2}]=[v_{k+2}yx]+[v_{k+2}xy],\\
&[v_{k}xyxv_{2}]=2([v_{k+2}yx^2]+[v_{k+2}xyx]).
\end{align*}
\end{lemma}

Assume the diamond $L_m$ has infinite type and the diamond $L_{m+q-1}$ has finite type $\mu$.
Set $v_{k+2}=[v_{k+1}xyx^{q-3}]$, unless $\mu=0$, in which case set $v_{k+2}=[v_{k+1}yx^{q-2}]$.
According to Theorem~\ref{thm:distance},
the element
$y$ centralizes $L_{m+q},\ldots,L_{m+2q-4}$.
Assume $L_{m+2q-2}$ is a diamond
(which is a consequence of Theorem~~\ref{thm:distance}
only if $\mu\neq 1$).
In case $\mu=1$ we do assume here that $L_{m+2q-2}$ is a diamond.
We describe the adjoint action of $v_{2}$ on the elements close to the diamond $L_m$.

\begin{lemma}\label{lemma:v2ext}
Suppose $[v_{k}xy]=-[v_{k}yx]$  and  $[v_{k+1}yx]=(\mu^{-1}-1)[v_{k+1}xy]$ for some $\mu\in \F^{\ast}$.
Then we have
\begin{align*}
&[v_{k}^{-1}v_{2}]=-3\mu^{-1}v_{k+2}^{-1},\\
&[v_{k}v_{2}]=-2\mu^{-1} v_{k+2},\\
&[v_{k}xv_{2}]=-\mu^{-1} [v_{k+2}x],\\
&[v_{k}yv_{2}]=-\mu^{-1} [v_{k+2}y], \\
&[v_{k}xyv_{2}]=[v_{k+2}xy]+(2\mu^{-1}+1)[v_{k+2}yx],\\
&[v_{k}xyxv_{2}]=2[v_{k+2}xyx]+(3\mu^{-1}+2)[v_{k+2}yx^2].
\end{align*}
\end{lemma}
\begin{proof}
All the stated equations except the first one follow from~\cite[Lemma~4.7]{AviMat:earliest}. We expand
\[
[v_{k}^{-1}v_{2}]=[v_{k}^{-1}[v_{1}x]yx^{q-3}]+3[v_{k}[v_{1}x]yx^{q-4}]
-3[v_{k}y[v_{1}x]x^{q-4}]-6[v_{k}xy[v_{1}x]x^{q-5}].
\]
Since $[v_{k}^{-1}[v_{1}x]]=0=[v_{k}[v_{1}x]]$, we deduce
\[
[v_{k}^{-1}v_{2}]=-3[v_{k}y[v_{1}x]x^{q-4}]-6[v_{k}xy[v_{1}x]x^{q-5}]=-3\mu^{-1}v_{k+2}^{-1},
\]
as desired.
\end{proof}
\begin{rem}\label{rem:v2ext_mu_0}
In the excluded case $\mu=0$ we find $[v_{k}^{-1}v_{2}]=-3v_{k+2}^{-1}$,
$[v_{k}v_{2}]=-2v_{k+2}$, $[v_{k}xv_{2}]=-[v_{k+2}x]$,
$[v_{k}yv_{2}]=-[v_{k+2}y]$, $[v_{k}xyv_{2}]=2[v_{k+2}yx]$ and $[v_{k}xyxv_{2}]=3[v_{k+2}yx^2]$.
\end{rem}
Suppose now that $L_{m}$ is a diamond of type $1$ such that $[L_{m+q-2}y]=0$. Let $0\neq v_{b}\in L_{m-1}$ and set $v_{b+1}=[v_{b}xyx^{q-2}]$ and $v_{b+2}=[v_{b}xyx^{q-3}]$. 
As an immediate consequence of Lemmas~\ref{lemma:v1} and~\ref{lemma:v2} we have
\begin{lemma}\label{le:type_1}
Under the above assumptions we have
\begin{align*}
&[v_{b}v_{1}]=2v_{b+1}^{-1}  &&[v_{b}xv_{1}]=v_{b+1} \\
&[v_{b}xyv_{1}]=-[v_{b+1}y]  &&[v_{b}xyv_{2}]=0
\end{align*}
Also, when $L_{m+q}$ has infinite type we have  $[v_{b}v_{2}]=2v_{b+2}^{-1}$ and $[v_{b}xv_{2}]=v_{b+2}$.
\end{lemma}

\bibliography{References}

\end{document}